\documentclass[12pt]{amsart}
\usepackage{amsmath,amssymb,amsthm,tikz,tikz-cd,color,xcolor,mathdots,amscd}
\usepackage[centertableaux]{ytableau}
\usepackage[all]{xy}
\usepackage[margin=1in]{geometry}
\usepackage{placeins}
\usetikzlibrary{shapes,arrows,svg.path}
\usepackage{booktabs}

\usepackage[colorlinks=true, pdfstartview=FitV, linkcolor=blue, citecolor=blue, urlcolor=blue]{hyperref}	

\newcommand{\doi}[1]{\href{https://doi.org/#1}{\texttt{doi:#1}}}

\theoremstyle{plain}
\newtheorem{theorem}{Theorem}[section]
\newtheorem{proposition}[theorem]{Proposition}
\newtheorem{corollary}[theorem]{Corollary}
\newtheorem{lemma}[theorem]{Lemma}

\theoremstyle{definition}

\theoremstyle{remark}
\newtheorem{remark}[theorem]{Remark}
\newtheorem{example}[theorem]{Example}
\numberwithin{equation}{section}

\definecolor{darkgreen}{RGB}{0,180,0}

\definecolor{darkred}{rgb}{0.7,0,0} % darkred color
\newcommand{\defn}[1]{{\color{darkred}\emph{#1}}} % emphasis of a definition 

\definecolor{UQgold}{RGB}{196, 158, 54} % UQ gold
\definecolor{UQpurple}{RGB}{73, 7, 94} % UQ purple

\newcommand{\ZZ}{\mathbb{Z}}

\newcommand{\CC}{\mathbb{C}}
\newcommand{\cc}{\mathbf{c}}
\newcommand{\xx}{\mathbf{x}}
\newcommand{\yy}{\mathbf{y}}

\newcommand{\sym}[1]{S_{#1}}  % The symmetric group
\newcommand{\states}{\mathfrak{S}}  % The states of the model
\newcommand{\Gstates}{\mathfrak{G}}  % The states of the double Grothendieck model
\newcommand{\Dstates}{\mathfrak{D}}  % The states of the semidual model
\newcommand{\G}{G}  % (double) Grothendieck polynomial
\newcommand{\wt}{\operatorname{wt}}  % weight
\newcommand{\abs}[1]{\lvert #1 \rvert}
\newcommand{\GL}{\operatorname{GL}}
\newcommand{\svt}{\operatorname{SVT}}  % Set-valued tableaux
\newcommand{\eyd}{\operatorname{EYD}}  % Excited Young diagrams
\newcommand{\End}{\operatorname{End}}  % End space

\newcommand{\Kdd}{\mathfrak{d}}%{\widetilde{\partial}}  % K-theoretic divided difference operator
\usepackage{listings}
\lstdefinelanguage{Sage}[]{Python}
{morekeywords={False,sage,True},sensitive=true}
\definecolor{dblackcolor}{rgb}{0.0,0.0,0.0}
\definecolor{dbluecolor}{rgb}{0.01,0.02,0.7}
\definecolor{dgreencolor}{rgb}{0.2,0.4,0.0}
\definecolor{dgraycolor}{rgb}{0.30,0.3,0.30}

\usepackage[colorinlistoftodos]{todonotes}

\begin{document}
\title{Double Grothendieck polynomials and colored lattice models}

\author{Valentin Buciumas}
\address[V.~Buciumas]{School of Mathematics and Physics, 
The University of Queensland, 
St.\ Lucia, QLD, 4072, 
Australia}
\email{valentin.buciumas@gmail.com}
\urladdr{https://sites.google.com/site/valentinbuciumas/}

\author{Travis Scrimshaw}
\address[T.~Scrimshaw]{School of Mathematics and Physics, 
The University of Queensland, 
St.\ Lucia, QLD, 4072, 
Australia}
\email{tcscrims@gmail.com}
\urladdr{https://people.smp.uq.edu.au/TravisScrimshaw/}


\begin{abstract}
We construct an integrable colored six-vertex model whose partition function is a double Grothendieck polynomial.
This gives an integrable systems interpretation of bumpless pipe dreams and recent results of Weigandt relating double Grothendieck polynomias with bumpless pipe dreams.
For vexillary permutations, we then construct a new model that we call the semidual version model.
We use our semidual model and the five-vertex model of Motegi and Sakai to given a new proof that double Grothendieck polynomials for vexillary permutations are equal to flagged factorial Grothendieck polynomials.
Taking the stable limit of double Grothendieck polynomials, we obtain a new proof that the stable limit is a factorial Grothendieck polynomial as defined by McNamara.
The states of our semidual model naturally correspond to families of nonintersecting lattice paths, where we can then use the Lindstr\"om--Gessel--Viennot lemma to give a determinant formula for double Schubert polynomials corresponding to vexillary permutations.
\end{abstract}

\maketitle

\section{Introduction}
\label{sec:intro}

The Yang--Baxter equation, also known as the star--triangle equation, has been the cornerstone of many aspects of modern mathematical physics.
It plays an central role in (quantum) integrable systems, which naturally arise across a broad spectrum of mathematics and physics.
One application has been to explain combinatorial phenomena and properties, such as in~\cite{BBBG19,BBF11,BSW20,BorodinWheelernsMac,EKLP92,FK94,FK96,GK17,HK05,HKZJ18,HPW20,Kuperberg96,KZJ17,MS13,MS14,MS19,WZJ19}.
Another has been to study probabilistic models~\cite{Borodin17,CP16,KMO15,KMO16,KMO16II,MS13,MS14}.
Solvable lattice models have also seen applications in the study of Whittaker functions~\cite{BBB,BBBG19II,BBBGIwahori,BBCFG12,BBCG12,Gray17,Ivanov12}.
In many of the papers cited, one equates an object of interest (such as a Hall--Littlewood polynomial or a Whittaker function) with the partition function of a solvable lattice model.
This then naturally implies interesting properties such as branching rules, Cauchy-type identities, exchange relations under Hecke operators, and combinatorial descriptions of the functions.

An important class of objects in algebraic geometry are Schubert varieties.
Consider the general linear group $G = \GL(\CC^n)$, the subgroup of lower triangular matrices $B$, and the subgroup of diagonal matrices $T$.
A complete flag is an element in the (complete/full) flag variety $G / B$ and corresponds to a sequence of subspaces $\{0\} = V_0 \subseteq \cdots \subseteq V_n = \CC^n$ such that $\dim V_i = i$.
A Schubert cell is a (left) $B$-orbit in $G / B$, and a Schubert variety is the Zariski closure of a Schubert cell.
The Schubert varieties are indexed by permutations and have many nice properties, such as Bruhat order corresponding to inclusion and forming a basis for the ($T$-)equivariant (connective) K-theory ring of the flag variety $K_T^{\beta}(G/B)$.
For additional background, we refer the reader to~\cite{Anderson11,Brion04} and references therein.

In this paper, we connect the equivariant K-theory of the flag variety with solvable lattice models.
To do this, we use the double ($\beta$-)Grothendieck polynomials $\G_w$, where $w$ is a permutation, as representatives for the Schubert varieties in the equivariant K-theory ring $K_T^{\beta}(G/B)$~\cite{FK94,FK96,Hudson14,LS82}.
Our main result (Theorem~\ref{thm:5vertex_atoms}) is a colored six-vertex model whose partition function is (up to a trivial factor of the parameter $\beta$) a double Grothendieck polynomial.
We construct our colored model as a translation of the bumpless pipe dreams with a fixed key given in~\cite{Weigandt20}, where resolving multiple crossings of two strands precisely corresponds to the colorization performed in~\cite{BBBG19,BSW20}.
Thus, we encode the permutation into the model by using the coloring.
Our proof is showing the Yang--Baxter equation implies the partition function satisfies the same functional equation as applying a divided difference operator defining double Grothendieck polynomials (up to a $\beta$ factor).
Therefore, we have a new proof of~\cite[Thm.~1.1]{Weigandt20}, which gives a formula for $\G_w$ as a sum over bumpless pipe dreams with key $w$ proven using a combination of algebraic and combinatorial techniques.
Our model is the colored version of the six-vertex model that Lascoux used in~\cite{Lascoux02} to describe $\G_w$ using alternating sign matrices.
By specializing $\beta = 0$, we have a new proof of the formula for double Schubert polynomials from~\cite{LLS18} in terms of bumpless pipe dreams.
We note that our results can be considered as the flag variety version of~\cite{BSW20,GK17,Motegi20,MS13} on the Grassmannian, the set of $k$-dimensional planes in $\CC^n$.

In the second part of our paper, we consider the analog of double Grothendieck polynomials for the equivariant K-theory of the Grassmannian, which are the factorial Grothendieck polynomials.
The stable limit of double Grothendieck polynomials decompose into finitely many factorial Grothendieck polynomials~\cite{Buch02,McNamara06}.
When $w$ is a vexillary permutation (it avoids the pattern $2143$), the stable limit of $\G_w$ is a single factorial Grothendieck polynomial.
Factorial Grothendieck polynomials have a determinant formula~\cite{IN13} and are given as the sum over set-valued tableaux~\cite{McNamara06}.

We look at what happens to our colored model restricted to vexillary permutations.
To do this, we first construct a variation of the uncolored version of our model by swapping $0 \leftrightarrow 1$ on the horizontal lines (\textit{i.e.}, on the auxiliary space), which we call the semidual model.
To encode the permutation $w$, we modify the semidual model to not be on a rectangular grid, but instead on a grid given by a partition $\Lambda_w$ that depends on $w$.
By a key property of vexillary permutations, the partition function for the semidual model is the same as for our original colored model and the semidual model becomes a five-vertex model.
In doing so, we can change the weights of one of the vertices by removing a $\beta$ factor and maintain integrability, and the resulting vertices become those of~\cite{GK17,MS13,Motegi20} (up to a gauge transformation and a symmetry of the model).
The fact that the shape of the semidual model is $\Lambda_w$ corresponds to imposing a flagging on set-valued tableaux under the natural bijection between (marked) states of the~\cite{MS13} model given in~\cite[Sec.~4.2]{MPS18}.
Therefore we obtain a new proof that double Grothendieck polynomials are sums over flagged set-valued tableaux~\cite{KMY09} (Theorem~\ref{thm:factorial_model}), which was proven using Gr\"obner geometry.
Furthermore, by taking the stable limit, we recover~\cite{McNamara06} (Theorem~\ref{thm:vexillary_flagged}), which was proven by examining the underlying combinatorics.
By taking appropriate specializations, we obtain new proofs of results from~\cite{Buch02,GK17,MS13,Motegi20,WZJ19}.

A state of the semidual model can also be realized as a family of nonintersecting lattice paths.
We can then apply the Lindstr\"om--Gessel--Viennot lemma~\cite{Lindstrom73,GV85}, but we can only do so when restricted to $\Lambda_w$ at $\beta = 0$, which gives a determinant formula for flagged (factorial) Schur functions (Theorem~\ref{thm:LGV_Schubert}) that seems to differ from~\cite{Wachs85} by elementary row and column operators.
However, we were unable to extend this to the determinant formula for $\G_w$ from~\cite{Anderson19,HM18,Matsumura17,MS19}.
Additionally, our nonintersecting lattice paths are distinct from those in~\cite{Kreiman05,LRS06} referenced in~\cite[Sec.~7.3]{Weigandt20} and from those in~\cite{Krattenthaler01,Krattenthaler05} and~\cite{KL04}.
We can also translate states of our models to excited Young diagrams, introduced independently in~\cite{Kreiman05} and~\cite{IN09}, by using the local moves from~\cite[Sec.~7]{Weigandt20}, which are also easy to see directly on the semidual model.
Excited Young diagrams were used to describe polynomial representatives of the $T$-equivariant K-theory of the Grassmannian~\cite{IN13,GK15}, and so we recover special cases of those results for $\GL(\CC^n)$.
We expect our approach using colored lattice models can be used to give a new proof of the general formulas from~\cite{IN13,GK15}.

This paper is organized as follows.
In Section~\ref{sec:background}, we provide the necessary background on double Grothendieck polynomials and colored lattice models.
In Section~\ref{sec:double_Grothendieck_model}, we give a new integrable colored lattice model whose partition function is a double Grothendieck polynomial and connect it to bumpless pipe dreams.
In Section~\ref{sec:dual_model}, we give our semidual model and connect it with excited Young diagrams and (flagged) set-valued tableaux.

After this paper was written, we were made aware of the paper~\cite{frozenpipes}, where a different lattice model interpretation for double Grothendieck polynomials is given.
The admissible states in their model are related to regular pipe dreams, whereas the admissible states in our model are related to bumpless pipes dreams.
The Boltzmann weights for the states (and the $L$-matrices) are different as there is no weight preserving bijection between the two.
Indeed, this can be seen for the case of $w = s_2$, where each of the two states have different factorizations (cf.~\cite[Ex.~6.3]{Weigandt20}).
Moreover, their applications are distinct from ours with the exception of Corollary~\ref{cor:symmetry}, which we realized we could prove after seeing their preprint.

%%%%%%%%%%%%%%%%%%%%
\subsection*{Acknowledgments}

The authors thank Anna Weigandt for useful discussions and explanations involving her recent work, as well as comments on an early draft of this paper.
The authors also thank Paul Zinn-Justin for useful discussions.
The authors are grateful to the authors of~\cite{frozenpipes} for sharing their preprint. 
TS thanks Kohei Motegi and Kazumitsu Sakai for valuable discussions on their results~\cite{MS13,MS14}.
TS also thanks Kohei Motegi for inspiration for the semidual model in Section~\ref{sec:dual_model}. 
This paper benefited from computations using \textsc{SageMath}~\cite{sage}.

VB was supported by the Australian Research Council DP180103150.

%===============================================================================
\section{Background}
\label{sec:background}

Fix positive integers $n$ and $m$.
Let $\xx = (x_1, x_2, \dotsc, x_n)$ be a finite sequence of indeterminates, and let $\yy = (y_1, y_2, \ldots)$ be an infinite sequence of indeterminates.
For any sequence $(\alpha_1, \dotsc, \alpha_n) \in \ZZ^n$ of length $n$, denote $\xx^{\alpha} := x_1^{\alpha_1} x_2^{\alpha_2} \cdots x_n^{\alpha_n}$.

Let $\lambda = (\lambda_1, \lambda_2, \dotsc, \lambda_n)$ be a \defn{partition}, a sequence of weakly decreasing nonnegative integers (of length $n$).
Let $\ell(\lambda) = \max \{k \mid \lambda_k > 0 \}$ denote the \defn{length} of $\lambda$.
The Young diagram (in English convention) of $\lambda$ is a drawing consisting of stacks of boxes with row $i$ having $\lambda_i$ boxes pushed into the upper-left corner.
The \defn{$01$-sequence} of $\lambda$ is given by reading the boundary of $\lambda$, starting at the bottom, with horizontal steps being $0$ and vertical steps being $1$ (we ignore all trailing $0$s in the $01$-sequence).
For example, the $01$-sequence of $\lambda = (5, 2, 2, 1, 0, 0)$ for $n = 6$ is $11010110001$.

\subsection{Permutations}

Let $\sym{n}$ denote the symmetric (or permutation) group on $n$ elements with simple transpositions $(s_1, s_2, \dotsc, s_{n-1})$.
For $w \in \sym{n}$, let $\ell(w)$ denote the length of $w$: the minimal number of simple transpositions whose product equals $w$.
We denote by $w_0$ the longest element in  $\sym{n}$.
The \defn{diagram} of $w$ is
\[
D(w) := \{ (p, q) \in \{1, \dotsc, n\} \times \{1, \dotsc, n\} \mid w(p) > q \text{ and } w^{-1}(q) > p \},
\]
(these are the boxes not in the Rothe diagram of $w$) and the \defn{essential set} of $w$ is
\[
E(w) := \{ (p, q) \in D(w) \mid (p+1, q), (p, q+1) \notin D(w) \}.
\]
Let $w' = 1^k \times w$ denote the permutation given by
\[
w'(i) := \begin{cases}
i & \text{if } i \leq k, \\
w(i-k) + k & \text{if } i > k.
\end{cases}
\]
Let $\leq$ denote the (strong) Bruhat order on $S_n$.
For more information on the symmetric group and Bruhat order, we refer the reader to~\cite{Sagan01}.
Define $w\xx := (x_{w(1)}, x_{w(2)}, \dotsc, x_{w(n)})$.

A \defn{vexillary permutation} is a permutation $w$ that avoids the pattern $2143$; that is to say there does not exists $1 \leq i < j < k < \ell \leq n$ such that $w_j < w_i < w_{\ell} < w_k$.

\begin{proposition}[{\cite[Sec.~9]{Fulton92}; see also~\cite[Cor.~3.3]{KMY09}}]
\label{prop:vexillary_classification}
The following are equivalent for a permutation $w$:
\begin{itemize}
\item The permutation $w$ is vexillary.
\item There does not exist $(p,q), (i, j) \in E(w)$ such that $p < i$ and $q < j$.
\item We can permute the rows and columns of $D(w)$ to obtain a partition $\lambda_w$.
\end{itemize}
\end{proposition}

Note that we can compute $\lambda_w$ by ordering the sequence whose $i$-th value is the number of boxes in row $i$ of $D(w)$.
We call $\lambda_w$ the \defn{partition associated to $w$}.
Next, let $\Lambda_w$ denote the smallest partition whose Young diagram contains the boxes $D(w)$.
Following~\cite[Sec.~5.2]{KMY09} (see also~\cite{Fulton92}), we define a sequence $F_w = (F_i)_{i=1}^{\ell(\lambda_w)}$ by $F_i$ denoting the row index of the southeastern box of $\Lambda_w$ that lies on the same diagonal as the last box $(i, \lambda_i)$ in row $i$ of $\lambda = \lambda_w$.
We call $F_w$ the \defn{flagging associated to $w$}.

%%%%%%%%%%%%%%%%%%%%
\subsection{Divided difference operators}

Define $x \oplus y := x + y + \beta x y$.
Let $f \in \ZZ[\xx]$.
The \defn{$i$-th (Newton) divided difference operator} is defined as
\[
\partial_i f := \dfrac{f - s_i f}{x_i - x_{i+1}},
\]
where we recall that $s_i$ acts by swapping the variables $x_i \leftrightarrow x_{i+1}$.
The \defn{K-theoretic divided difference operator} is
\[
\Kdd_i f := \partial_i \bigl( (1 + \beta x_{i+1}) f \bigr) = \frac{(1 + \beta x_{i+1}) f - (1 + \beta x_i) s_i f}{x_i - x_{i+1}}.
\]
These are normally denoted by $\pi_i$ in the literature, but we choose a different notation in order to avoid confusion with the common notation for the  \defn{$i$-th Demazure operator}, which is defined as
\[
\pi_i f := \partial_i (x_i f) = \frac{x_i f - x_{i+1} s_i f}{x_i - x_{i+1}}.
\]
We also require the \defn{Demazure--Lascoux operator} and the \defn{Demazure--Lascoux atom operator}
\[
\varpi_i f := \pi_i\bigl( (1 + \beta x_{i+1}) f \bigr),
\hspace{60pt}
\overline{\varpi}_i := \varpi_i - 1.
\]
For any operator $D_i = \partial_i, \Kdd_i, \pi_i, \varpi_i, \overline{\varpi}_i$, define
$
D_w := D_{i_1} \dotsm D_{i_k}
$
for any reduced expression $w = s_{i_1} \cdots s_{i_k}$, which is well-defined since $D_i$ satisfies the braid relations:
\[
D_i D_{i+1} D_i = D_{i+1} D_i D_{i+1},
\qquad\qquad
D_i D_j = D_j D_i
\quad \text{for } \abs{i - j} > 1.
\]
%We also note the following relations:
%\[
%\partial_i^2 = \Kdd_i^2 =  \overline{\varpi}_i^2 =0,
%\qquad\qquad
%\pi_i^2 = \pi_i,
%\qquad\qquad
%\varpi_i^2 = \varpi_i.
%\]

%%%%%%%%%%%%%%%%%%%%
\subsection{Grothendieck polynomials}

Let $w \in S_n$ be a permutation.
Following~\cite{FK94}, the \defn{double Grothendieck polynomial} is defined recursively by
\[
\G_{w_0}(\xx, \yy; \beta) := \prod_{i + j \leq n} x_i \oplus y_j,
\qquad\quad
\G_w(\xx, \yy; \beta) := \Kdd_i \G_{w s_i}(\xx, \yy; \beta), \quad \text{for } \ell(w s_i) = \ell(w) + 1.
\]

A \defn{set-valued tableau of shape $\lambda$} is a filling of the boxes of $\lambda$ with non-empty (finite) subsets of $\{1, 2, \dotsc, n\}$ that are weakly increasing along rows and strictly increasing along columns in the following sense:
\[
\ytableaushort{AB,C}\,,
\qquad\qquad
\max A \leq \min B,
\qquad
\max A < \min C.
\]
Let $\svt_{\lambda}$ denote the set of set-valued tableaux of shape $\lambda$ (where each entry is a subset of $\{1, \dotsc, n\}$).
For $T$ a set-valued tableau, we write $A \in T$ to mean that $A$ is an entry (which is a set) in one of the boxes of $T$.
Let $F = \{F_i\}_{i=1}^{\ell(\lambda)}$ be a (finite) sequence of integers.
Let $\svt_{\lambda,F}$ denote the subset of set-valued tableaux of shape $\lambda$ such that every value in the $i$-th row of $T$ is at most $F_i$.
We call a set-valued tableau in $\svt_{\lambda,F}$ a \defn{flagged set-valued tableau}.

Following~\cite{McNamara06}, the \defn{factorial Grothendieck polynomial} is defined by
\[
\G_{\lambda}(\xx | \yy; \beta) := \sum_{T \in \svt_{\lambda}} \beta^{\abs{T} - \abs{\lambda}} \prod_{A \in T} \prod_{i \in A} x_i \oplus y_{i + c(A)},
\]
where $c(A) = c - r$ is the \defn{content} of the box $A$ (which we have equated with its entry) in row $r$ and column $c$.
A determinant formula for the factorial Grothendieck polynomials was given in~\cite{IN13,IS14}.
Following~\cite{KMY09}, define the \defn{flagged factorial Grothendieck polynomial} as
\[
\G_{\lambda,F}(\xx | \yy; \beta) := \sum_{T \in \svt_{\lambda,F}} \beta^{\abs{T} - \abs{\lambda}} \prod_{A \in T} \prod_{i \in A} x_i \oplus y_{i + c(A)}.
\]
A determinant formula for $\G_{\lambda,F}(\xx, \yy; \beta)$ was given in~\cite{MS19}.

\begin{theorem}[{\cite[Thm.~5.8]{KMY09}}]
\label{thm:vexillary_flagged}
Let $w$ be a vexillary permutation. Then we have
\[
\G_w(\xx, \yy; \beta) = \G_{\lambda_w,F_w}(\xx | \yy; \beta).
\]
\end{theorem}

When we take $\beta = 0$ in Theorem~\ref{thm:vexillary_flagged}, we recover the corresponding result in~\cite{LS82} for Schubert polynomials (and $\yy \mapsto -\yy$) via a Jacobi--Trudi-type determinant formula (see also~\cite{Wachs85,KM05}).

%%%%%%%%%%%%%%%%%%%%
\subsection{Colored models}
\label{sec:colored_model}

We give the colored model from~\cite{BSW20} but with the double Grothendieck substitution $x \mapsto x \oplus y$ (\textit{i.e.}, we perform a gauge transformation).

\begin{figure}
\[
\begin{array}{c@{\;\;}c@{\;\;}c@{\;\;}c@{\;\;}c@{\;\;}c@{\;\;}c}
\toprule
  \tt{u}_1&\tt{u}_2&\tt{u}^{\dagger}_2&\tt{u}^{\circ}_2&\tt{v}_2&\tt{w}_1&\tt{w}_2\\
\midrule
\begin{tikzpicture}
\coordinate (a) at (-.75, 0);
\coordinate (b) at (0, .75);
\coordinate (c) at (.75, 0);
\coordinate (d) at (0, -.75);
\coordinate (aa) at (-.75,.5);
\coordinate (cc) at (.75,.5);
\draw (a)--(0,0);
\draw (b)--(0,0);
\draw (c)--(0,0);
\draw (d)--(0,0);
\draw[fill=white] (a) circle (.25);
\draw[fill=white] (b) circle (.25);
\draw[fill=white] (c) circle (.25);
\draw[fill=white] (d) circle (.25);
\node at (0,1) { };
\node at (a) {$0$};
\node at (b) {$0$};
\node at (c) {$0$};
\node at (d) {$0$};
\path[fill=white] (0,0) circle (.2);
\node at (0,0) {$x$};
\end{tikzpicture}
%%%%%%%
& \begin{tikzpicture}
\coordinate (a) at (-.75, 0);
\coordinate (b) at (0, .75);
\coordinate (c) at (.75, 0);
\coordinate (d) at (0, -.75);
\coordinate (aa) at (-.75,.5);
\coordinate (cc) at (.75,.5);
\draw[line width=0.5mm, blue] (c)--(0,0);
\draw[line width=0.5mm, blue] (b)--(0,0);
\draw[line width=0.5mm, red] (a)--(0,0);
\draw[line width=0.5mm, red] (d)--(0,0);
\draw[line width=0.5mm, blue,fill=white] (c) circle (.25);
\draw[line width=0.5mm, blue,fill=white] (b) circle (.25);
\draw[line width=0.5mm, red, fill=white] (a) circle (.25);
\draw[line width=0.5mm, red, fill=white] (d) circle (.25);
\node at (0,1) { };
\node at (c) {$c'$};
\node at (b) {$c'$};
\node at (a) {$c$};
\node at (d) {$c$};
\path[fill=white] (0,0) circle (.2);
\node at (0,0) {$x$};
\end{tikzpicture}
%%%%%%%
& \begin{tikzpicture}
\coordinate (a) at (-.75, 0);
\coordinate (b) at (0, .75);
\coordinate (c) at (.75, 0);
\coordinate (d) at (0, -.75);
\coordinate (aa) at (-.75,.5);
\coordinate (cc) at (.75,.5);
\draw[line width=0.5mm, blue] (a)--(0,0);
\draw[line width=0.6mm, red] (b)--(0,0);
\draw[line width=0.5mm, blue] (c)--(0,0);
\draw[line width=0.6mm, red] (d)--(0,0);
\draw[line width=0.5mm, blue,fill=white] (a) circle (.25);
\draw[line width=0.5mm, red,fill=white] (b) circle (.25);
\draw[line width=0.5mm, blue, fill=white] (c) circle (.25);
\draw[line width=0.5mm, red, fill=white] (d) circle (.25);
\node at (0,1) { };
\node at (a) {$c'$};
\node at (b) {$c$};
\node at (c) {$c'$};
\node at (d) {$c$};
\path[fill=white] (0,0) circle (.2);
\node at (0,0) {$x$};
\end{tikzpicture}
%%%%%%%
& \begin{tikzpicture}
\coordinate (a) at (-.75, 0);
\coordinate (b) at (0, .75);
\coordinate (c) at (.75, 0);
\coordinate (d) at (0, -.75);
\coordinate (aa) at (-.75,.5);
\coordinate (cc) at (.75,.5);
\draw[line width=0.5mm, UQgold] (a)--(0,0);
\draw[line width=0.6mm, UQgold] (b)--(0,0);
\draw[line width=0.5mm, UQgold] (c)--(0,0);
\draw[line width=0.6mm, UQgold] (d)--(0,0);
\draw[line width=0.5mm, UQgold,fill=white] (a) circle (.25);
\draw[line width=0.5mm, UQgold,fill=white] (b) circle (.25);
\draw[line width=0.5mm, UQgold, fill=white] (c) circle (.25);
\draw[line width=0.5mm, UQgold, fill=white] (d) circle (.25);
\node at (0,1) { };
\node at (a) {$d$};
\node at (b) {$d$};
\node at (c) {$d$};
\node at (d) {$d$};
\path[fill=white] (0,0) circle (.2);
\node at (0,0) {$x$};
\end{tikzpicture}
%%%%%%%
& \begin{tikzpicture}
\coordinate (a) at (-.75, 0);
\coordinate (b) at (0, .75);
\coordinate (c) at (.75, 0);
\coordinate (d) at (0, -.75);
\coordinate (aa) at (-.75,.5);
\coordinate (cc) at (.75,.5);
\draw[line width=0.5mm, UQgold] (a)--(0,0);
\draw(b)--(0,0);
\draw[line width=0.5mm, UQgold] (c)--(0,0);
\draw (d)--(0,0);
\draw[line width=0.5mm,UQgold,fill=white] (a) circle (.25);
\draw[fill=white] (b) circle (.25);
\draw[line width=0.5mm,UQgold,fill=white] (c) circle (.25);
\draw[fill=white] (d) circle (.25);
\node at (0,1) { };
\node at (a) {$d$};
\node at (b) {$0$};
\node at (c) {$d$};
\node at (d) {$0$};
\path[fill=white] (0,0) circle (.2);
\node at (0,0) {$x$};
\end{tikzpicture}
%%%%%%%
& \begin{tikzpicture}
\coordinate (a) at (-.75, 0);
\coordinate (b) at (0, .75);
\coordinate (c) at (.75, 0);
\coordinate (d) at (0, -.75);
\coordinate (aa) at (-.75,.5);
\coordinate (cc) at (.75,.5);
\draw[line width=0.5mm, UQgold] (c)--(0,0);
\draw[line width=0.6mm, UQgold] (b)--(0,0);
\draw (a)--(0,0);
\draw (d)--(0,0);
\draw[line width=0.5mm,UQgold,fill=white] (c) circle (.25);
\draw[line width=0.5mm,UQgold,fill=white] (b) circle (.25);
\draw[fill=white] (a) circle (.25);
\draw[fill=white] (d) circle (.25);
\node at (0,1) { };
\node at (c) {$d$};
\node at (b) {$d$};
\node at (a) {$0$};
\node at (d) {$0$};
\path[fill=white] (0,0) circle (.2);
\node at (0,0) {$x$};
\end{tikzpicture}
%%%%%%%
& \begin{tikzpicture}
\coordinate (a) at (-.75, 0);
\coordinate (b) at (0, .75);
\coordinate (c) at (.75, 0);
\coordinate (d) at (0, -.75);
\coordinate (aa) at (-.75,.5);
\coordinate (cc) at (.75,.5);
\draw (c)--(0,0);
\draw (b)--(0,0);
\draw[line width=0.5mm, UQgold] (a)--(0,0);
\draw[line width=0.5mm, UQgold] (d)--(0,0);
\draw[fill=white] (c) circle (.25);
\draw[fill=white] (b) circle (.25);
\draw[line width=0.5mm,UQgold, fill=white] (a) circle (.25);
\draw[line width=0.5mm,UQgold, fill=white] (d) circle (.25);
\node at (0,1) { };
\node at (c) {$0$};
\node at (b) {$0$};
\node at (a) {$d$};
\node at (d) {$d$};
\path[fill=white] (0,0) circle (.2);
\node at (0,0) {$x$};
\end{tikzpicture}
%%%%%%%
\\ 
\midrule
1 & 1 &1 & 1 & x \oplus y & 1 & 1+\beta (x \oplus y) \\
\bottomrule
\end{array}
\]
\caption{The colored Boltzmann weights with ${\color{red} c} > {\color{blue} c'}$ and ${\color{UQgold} d}$ being any color.}
\label{fig:colored_weights}
\end{figure}

We construct a \defn{(lattice) model} by first considering a rectangular grid of $n$ horizontal lines and $m$ vertical lines.
Each crossing of two lines will be a \defn{vertex} and the lines between two vertices (resp.\ from one vertex) are \defn{edges} (resp.\ \defn{half edges}).
Fix an $n$-tuple of colors $\cc = (c_1 > c_2 > \cdots > c_n > 0)$ and a permutation $w \in \sym{n}$.
Let $w \cc = (c_{w(1)}, c_{w(2)}, \dotsc, c_{w(n)})$ be the natural action of $w$ on the colors.
The \defn{boundary conditions}, a labeling of the half edges, are the bottom and left half edges are labeled by $0$, the right half edges are labeled by $w \cc$ from top-to-bottom, and the top half edges given by $\lambda$ with the $i$-th $1$ in the $01$-sequence of $\lambda$, counted from the left, having color $c_i$.
See~\cite[Fig.~6]{BSW20} for an example of a state with given boundary conditions, though note that in~\cite{BSW20}, the left boundary is colored as opposed to the right one. 
A \defn{state} is a labeling of the edges by $\{0\} \sqcup \cc$, and we call a state \defn{admissible} if the local configurations around each vertex are one of the configurations given by Figure~\ref{fig:colored_weights}.

We assign a non-zero \defn{(Boltzmann) weight} with \defn{spectral parameter} $x$ to each of the vertices as in Figure~\ref{fig:colored_weights}, and any other vertex will have a Boltzmann weight of $0$.
The \defn{(Boltzmann) weight} $\wt(S)$ of a state $S$ is the product of all of the Boltzmann weights of all vertices with $x = x_i$ in the $i$-th row numbered starting from top and $y = y_j$ in the $j$-th column numbered starting from the left.
Note that a state $S$ is admissible if and only if $\wt(S) \neq 0$.
Let $\overline{\states}_{\lambda,w}$ denote the set of all possible admissible states for this model.
We will use $\overline{\states}_{\lambda,w}$ to refer to the model in our formulas.
The \defn{partition function} of the model $\overline{\states}_{\lambda,w}$ is
\[
Z(\overline{\states}_{\lambda,w}; \xx, \yy; \beta) := \sum_{S \in \overline{\states}_{\lambda,w}} \wt(S),
\]
the sum of the Boltzmann weights of all possible (admissible) states of the model.

\begin{remark}
\label{remark:gauge}
The Boltzmann weights of an admissible configuration around a vertex can be considered as an \defn{$L$-matrix} $L \in \End( W_a \otimes V_j )$, where $W_a = \CC^{n+1}$ and $V_j = \CC^{n+1}$ can be thought of as representations of a certain quantum group.
For more information on the relation between lattice models and quantum groups see~\cite[Ch.~7.5]{CPbook}. 
\end{remark}

From the description of the vertex weights, we can map a state in $\states_{\lambda,w}$ to a wiring diagram of $w$, where the different strands are represented by different colors.
Indeed, we can think of the configurations $\tt{v}_2$, $\tt{w}_1$, and $\tt{w}_2$ in Figure~\ref{fig:colored_weights} as a single strand passing through the vertex (possibly turning), $\tt{u}_2^{\dagger}$ as two strands crossing at the vertex (thus corresponding to a simple transposition), and $\tt{u}_1$ as two strands both passing near the vertex but not crossing.
Note that $\tt{u}_2^{\circ}$ does not appear in the model.
From the requirement ${\color{red} c} > {\color{blue} c'}$ in $\tt{u}_2^{\dagger}$, we see that any two strands can only cross at most once.

\begin{remark}
In our model, we color the right side and follow the convention of the model in~\cite{BBBG19} (see also~\cite[Rem.~3.1]{BSW20}).
Note that the weights are reflected along the vertical axis from~\cite{BSW20}, which also has the left side colored.
The model we recall here also differs from~\cite{BSW20} by $w \mapsto w_0 w$.
Thus $\overline{\states}_{\lambda, w}$ is the wiring diagram of $w$ (unlike in~\cite{BSW20}, where it corresponds to $w_0 w$).
We do this in order to match the conventions of Grothendieck polynomials and bumpless pipe dreams from~\cite{Weigandt20}.
\end{remark}

\begin{remark}
\label{rem:specialized_models}
Note that when $\yy = 0$, we have the colored model from~\cite{BSW20}, which recovers the model of~\cite{BBBG19} when we additionally set $\beta = 0$.
If we restrict to a single color, which we call the \defn{uncolored model}, then our $L$-matrix is the one from~\cite[Sec.~2]{Motegi20} with an appropriate gauge transformation and reflected over the vertical axis at $q = 0$, which is also the $L'$-matrix given in~\cite[Fig.~3.1]{GK17} after swapping $1 \leftrightarrow 0$.
This further restricts to the model of Motegi--Sakai~\cite{MS13} (also with an appropriate gauge transformation) by taking $\yy = 0$.
\end{remark}

\begin{figure}%[h]
\[
\begin{array}{c@{\hspace{20pt}}c@{\hspace{20pt}}c@{\hspace{20pt}}c}
\toprule
%  ++++ : z2
\begin{tikzpicture}[scale=0.7]
\draw (0,0) to [out = 0, in = 180] (2,2);
\draw (0,2) to [out = 0, in = 180] (2,0);
\draw[fill=white] (0,0) circle (.35);
\draw[fill=white] (0,2) circle (.35);
\draw[fill=white] (2,0) circle (.35);
\draw[fill=white] (2,2) circle (.35);
\node at (0,0) {$0$};
\node at (0,2) {$0$};
\node at (2,2) {$0$};
\node at (2,0) {$0$};
\path[fill=white] (1,1) circle (.3);
\node at (1,1) {$x_i,x_j$};
\end{tikzpicture}&
%  +RR+ : z2
\begin{tikzpicture}[scale=0.7]
\draw (0,0) to [out = 0, in = 180] (2,2);
\draw (0,2) to [out = 0, in = 180] (2,0);
\draw[fill=white] (0,0) circle (.35);
\draw[line width=0.5mm, UQgold, fill=white] (0,2) circle (.35);
\draw[line width=0.5mm, UQgold, fill=white] (2,2) circle (.35);
\draw[fill=white] (2,0) circle (.35);
\node at (0,0) {$0$};
\node at (0,2) {$d$};
\node at (2,2) {$d$};
\node at (2,0) {$0$};
\path[fill=white] (1,1) circle (.3);
\node at (1,1) {$x_i,x_j$};
\end{tikzpicture}&
%  R++R : z1
\begin{tikzpicture}[scale=0.7]
\draw (0,0) to [out = 0, in = 180] (2,2);
\draw (0,2) to [out = 0, in = 180] (2,0);
\draw[fill=white] (0,2) circle (.35);
\draw[line width=0.5mm, UQgold, fill=white] (0,0) circle (.35);
\draw[fill=white] (2,0) circle (.35);
\draw[line width=0.5mm, UQgold, fill=white] (2,2) circle (.35);
\node at (0,0) {$d$};
\node at (0,2) {$0$};
\node at (2,2) {$d$};
\node at (2,0) {$0$};
\path[fill=white] (1,1) circle (.3);
\node at (1,1) {$x_i,x_j$};
\end{tikzpicture}&
%  R++R : z2
\begin{tikzpicture}[scale=0.7]
\draw (0,0) to [out = 0, in = 180] (2,2);
\draw (0,2) to [out = 0, in = 180] (2,0);
\draw[line width=0.5mm, UQgold, fill=white] (0,0) circle (.35);
\draw[fill=white] (0,2) circle (.35);
\draw[fill=white] (2,2) circle (.35);
\draw[line width=0.5mm, UQgold, fill=white] (2,0) circle (.35);
\node at (0,0) {$d$};
\node at (0,2) {$0$};
\node at (2,2) {$0$};
\node at (2,0) {$d$};
\path[fill=white] (1,1) circle (.3);
\node at (1,1) {$x_i,x_j$};
\end{tikzpicture}\\
   \midrule
 (1 + \beta Z_i) Z_j & (1+ \beta Z_i) Z_j & (Z_j - Z_i) Z_j & (1 + \beta Z_j) Z_j \\
   \midrule
%  BRRB : z1
\begin{tikzpicture}[scale=0.7]
\draw (0,0) to [out = 0, in = 180] (2,2);
\draw (0,2) to [out = 0, in = 180] (2,0);
\draw[line width=0.5mm, blue, fill=white] (0,0) circle (.35);
\draw[line width=0.5mm, red, fill=white] (0,2) circle (.35);
\draw[line width=0.5mm, red, fill=white] (2,2) circle (.35);
\draw[line width=0.5mm, blue, fill=white] (2,0) circle (.35);
\node at (0,0) {$c'$};
\node at (0,2) {$c$};
\node at (2,2) {$c$};
\node at (2,0) {$c'$};
\path[fill=white] (1,1) circle (.3);
\node at (1,1) {$x_i,x_j$};
\end{tikzpicture}&
%  RBBR : z2
\begin{tikzpicture}[scale=0.7]
\draw (0,0) to [out = 0, in = 180] (2,2);
\draw (0,2) to [out = 0, in = 180] (2,0);
\draw[line width=0.5mm, red, fill=white] (0,0) circle (.35);
\draw[line width=0.5mm, blue, fill=white] (0,2) circle (.35);
\draw[line width=0.5mm, blue, fill=white] (2,2) circle (.35);
\draw[line width=0.5mm, red, fill=white] (2,0) circle (.35);
\node at (0,0) {$c$};
\node at (0,2) {$c'$};
\node at (2,2) {$c'$};
\node at (2,0) {$c$};
\path[fill=white] (1,1) circle (.3);
\node at (1,1) {$x_i,x_j$};
\end{tikzpicture}&
%  RBRB : z1 - z2
\begin{tikzpicture}[scale=0.7]
\draw (0,0) to [out = 0, in = 180] (2,2);
\draw (0,2) to [out = 0, in = 180] (2,0);
\draw[line width=0.5mm, blue, fill=white] (0,0) circle (.35);
\draw[line width=0.5mm, red, fill=white] (0,2) circle (.35);
\draw[line width=0.5mm, blue, fill=white] (2,2) circle (.35);
\draw[line width=0.5mm, red, fill=white] (2,0) circle (.35);
\node at (0,0) {$c_h$};
\node at (0,2) {$c$};
\node at (2,2) {$c'$};
\node at (2,0) {$c$};
\path[fill=white] (1,1) circle (.3);
\node at (1,1) {$x_i,x_j$};
\end{tikzpicture}&
%  RRRR : z2
\begin{tikzpicture}[scale=0.7]
\draw (0,0) to [out = 0, in = 180] (2,2);
\draw (0,2) to [out = 0, in = 180] (2,0);
\draw[line width=0.5mm, UQgold, fill=white] (0,0) circle (.35);
\draw[line width=0.5mm, UQgold, fill=white] (0,2) circle (.35);
\draw[line width=0.5mm, UQgold, fill=white] (2,2) circle (.35);
\draw[line width=0.5mm, UQgold, fill=white] (2,0) circle (.35);
\node at (0,0) {$d$};
\node at (0,2) {$d$};
\node at (2,2) {$d$};
\node at (2,0) {$d$};
\path[fill=white] (1,1) circle (.3);
\node at (1,1) {$x_i,x_j$};
\end{tikzpicture}\\
   \midrule
 (1 + \beta Z_j) Z_i & (1 + \beta Z_i) Z_j & Z_j - Z_i & (1 + \beta Z_i) Z_j \\
   \bottomrule
\end{array}
\]
\caption{The colored $R$-matrix with ${\color{red} c} > {\color{blue} c'}$ and ${\color{UQgold} d}$ being any color, where $Z_i = x_i \oplus y$ and $Z_j = x_j \oplus y'$. Note that the weights are not symmetric with respect to color.}
\label{fig:colored_R_matrix}
\end{figure}
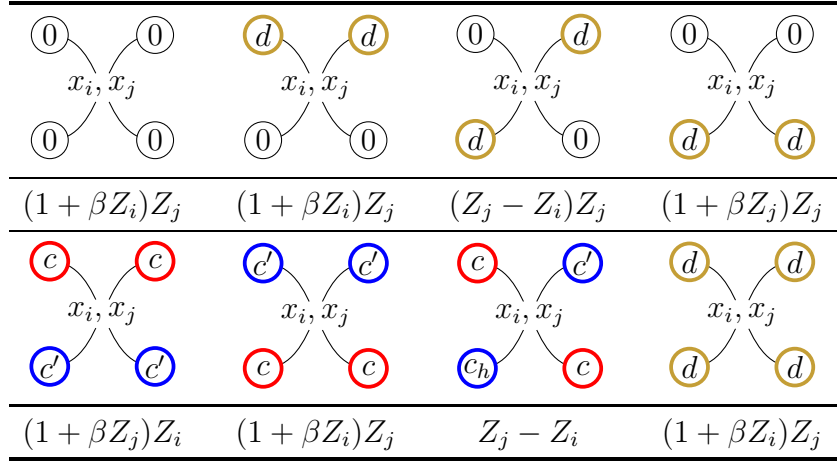

The $L$-matrix from Figure~\ref{fig:colored_weights} and the $R$-matrix from Figure~\ref{fig:colored_R_matrix} satisfy the $RLL$ form of the \defn{Yang--Baxter equation}.
Thus, this is an \defn{integrable} model.

\begin{proposition}[{\cite[Prop.~3.2]{BSW20}}]
\label{prop:YBE}
The partition function of the following two models are equal for any boundary conditions $a,b,c,d,e,f \in \{ 0, c_1, c_2, \dotsc, c_n \}$:
\begin{equation}
\label{eq:RLL_relation}
\begin{tikzpicture}[baseline=(current bounding box.center)]
  \draw (0,1) to [out = 0, in = 180] (2,3) to (4,3);
  \draw (0,3) to [out = 0, in = 180] (2,1) to (4,1);
  \draw (3,0) to (3,4);
  \draw[fill=white] (0,1) circle (.3);
  \draw[fill=white] (0,3) circle (.3);
  \draw[fill=white] (3,4) circle (.3);
  \draw[fill=white] (4,3) circle (.3);
  \draw[fill=white] (4,1) circle (.3);
  \draw[fill=white] (3,0) circle (.3);
  \node at (0,1) {$a$};
  \node at (0,3) {$b$};
  \node at (3,4) {$c$};
  \node at (4,3) {$d$};
  \node at (4,1) {$e$};
  \node at (3,0) {$f$};
\path[fill=white] (3,3) circle (.3);
\node at (3,3) {$x_i$};
\path[fill=white] (3,1) circle (.3);
\node at (3,1) {$x_j$};
\path[fill=white] (1,2) circle (.3);
\node at (1,2) {$x_i,x_j$};
\end{tikzpicture}\qquad\qquad
\begin{tikzpicture}[baseline=(current bounding box.center)]
  \draw (0,1) to (2,1) to [out = 0, in = 180] (4,3);
  \draw (0,3) to (2,3) to [out = 0, in = 180] (4,1);
  \draw (1,0) to (1,4);
  \draw[fill=white] (0,1) circle (.3);
  \draw[fill=white] (0,3) circle (.3);
  \draw[fill=white] (1,4) circle (.3);
  \draw[fill=white] (4,3) circle (.3);
  \draw[fill=white] (4,1) circle (.3);
  \draw[fill=white] (1,0) circle (.3);
  \node at (0,1) {$a$};
  \node at (0,3) {$b$};
  \node at (1,4) {$c$};
  \node at (4,3) {$d$};
  \node at (4,1) {$e$};
  \node at (1,0) {$f$};
\path[fill=white] (1,3) circle (.3);
\node at (1,3) {$x_j$};
\path[fill=white] (1,1) circle (.3);
\node at (1,1) {$x_i$};
\path[fill=white] (3,2) circle (.3);
\node at (3,2) {$x_i,x_j$};
\end{tikzpicture}
\end{equation}
\end{proposition}

\begin{remark}
\label{remark:finite_colors}
Note that the Yang--Baxter relation only involves at most $3$ colors, so Proposition~\ref{prop:YBE} reduces to an identity of $2^4 \times 2^4$ matrices since colors are conserved under the $R$-matrix and $L$-matrix.
\end{remark}

%Furthermore, we can see that the $R$-matrix corresponds to the vertices of the $L$-matrix rotated by $45^{\circ}$ clockwise and the weights of the $L$-matrix take $z = \frac{x_j - x_i}{1 + \beta x_i}$ and are scaled by $(1 + \beta x_i) x_j$.

\begin{figure}
\begin{tikzpicture}[scale=0.7]
\begin{scope}[shift={(6,0)}]
  \draw (4,1) to (6,1) to [out = 0, in = 180] (8,3);
  \draw (4,3) to (6,3) to [out = 0, in = 180] (8,1);
  \draw (0,1) to (2,1);
  \draw (0,3) to (2,3);
  \draw (5,0.25) to (5,3.75);
  \draw (1,0.25) to (1,3.75);
  \draw[fill=white] (-0.3,3) circle (.55);
  \draw[fill=white] (-0.3,1) circle (.55);
  \draw[line width=0.5mm,blue,fill=white] (8.3,1) circle (.55);
  \draw[line width=0.5mm,red,fill=white] (8.3,3) circle (.55);
  \node at (3,1) {$\cdots$};
  \node at (3,3) {$\cdots$};
  \draw[densely dashed] (1,3.75) to (1,4.25);
  \draw[densely dashed] (1,0.25) to (1,-0.25);
  \draw[densely dashed] (5,3.75) to (5,4.25);
  \draw[densely dashed] (5,0.25) to (5,-0.25);
  \path[fill=white] (1,3) circle (.5);
  \node at (1,3) {\scriptsize$x_i$};
  \path[fill=white] (1,1) circle (.4);
  \node at (1,1) {\scriptsize$x_{i+1}$};
  \path[fill=white] (5,3) circle (.5);
  \node (a) at (5,3) {\scriptsize$x_i$};
  \path[fill=white] (5,1) circle (.4);
  \node at (5,1) {\scriptsize$x_{i+1}$};
  \path[fill=white] (7,2) circle (.3);
  \node at (7,2) {\scriptsize$x_i,x_{i+1}$};
  \node at (8.3,1) {$d_i$};
  \node at (8.3,3) {$d_{i+1}$};
  \node at (-0.3,1) {$0$};
  \node at (-0.3,3) {$0$};
\end{scope}
\begin{scope}[shift={(-6,0)}]
  \draw (0,1) to [out = 0, in = 180] (2,3) to (4,3);
  \draw (0,3) to [out = 0, in = 180] (2,1) to (4,1);
  \draw (3,0.25) to (3,3.75);
  \draw (7,0.25) to (7,3.75);
  \draw (6,1) to (8,1);
  \draw (6,3) to (8,3);
  \draw[line width=0.5mm,blue,fill=white] (8.3,1) circle (.55);
  \draw[line width=0.5mm,red,fill=white] (8.3,3) circle (.55);
  \draw[fill=white] (-0.3,3) circle (.55);
  \draw[fill=white] (-0.3,1) circle (.55);
  \node at (-0.3,1) {$0$};
  \node at (-0.3,3) {$0$};
  \node at (5,3) {$\cdots$};
  \node at (5,1) {$\cdots$};
  \draw[densely dashed] (3,3.75) to (3,4.25);
  \draw[densely dashed] (3,0.25) to (3,-0.25);
  \draw[densely dashed] (7,3.75) to (7,4.25);
  \draw[densely dashed] (7,0.25) to (7,-0.25);
  \node at (8.3,1) {$d_i$};
  \node at (8.3,3) {$d_{i+1}$};
\path[fill=white] (3,3) circle (.4);
\node at (3,3) {\scriptsize$x_{i+1}$};
\path[fill=white] (3,1) circle (.5);
\node at (3,1) {\scriptsize$x_i$};
\path[fill=white] (7,3) circle (.4);
\node at (7,3) {\scriptsize$x_{i+1}$};
\path[fill=white] (7,1) circle (.5);
\node at (7,1) {\scriptsize$x_i$};
\path[fill=white] (1,2) circle (.3);
\node at (1,2) {\scriptsize$x_i,x_{i+1}$};
\end{scope}
\end{tikzpicture}
\caption{The standard train argument described pictorially.
   Left: The model with an $R$-matrix attached on the left.
   Right: The model after repeatedly using the Yang--Baxter equation~\eqref{eq:RLL_relation}.}
\label{fig:train_argument}
\end{figure}
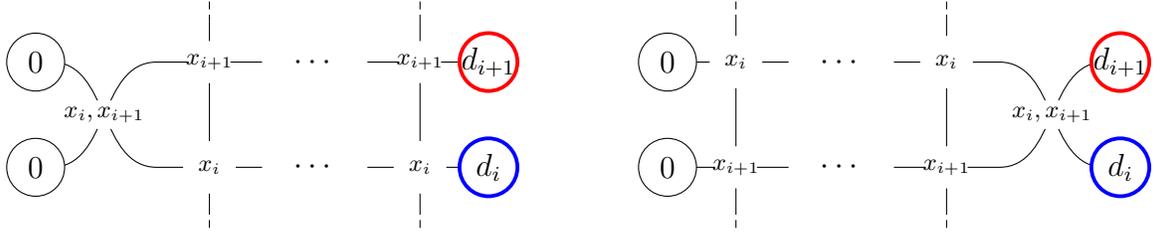

Using the integrability of the model and the standard train argument,\footnote{The standard train argument is simply repeatedly applying the Yang--Baxter equation to pass an $R$-matrix from the left to the right of a model with fixed boundary conditions; see Figure~\ref{fig:train_argument}.} we can derive a functional equation for the partition function as in~\cite[Lemma~3.3]{BSW20} using a fixed value $y_1 = y_2 = \cdots = y_m = y$ (which we succinctly write as $\yy = y$) in the left and right columns respectively:
\[
Z(\overline{\states}_{\lambda,s_i w}; \xx, y; \beta)
=
\dfrac{(1 + \beta x_i) (x_{i+1} \oplus y) \bigl( Z(\overline{\states}_{\lambda,w}; \xx, y; \beta) + Z(\overline{\states}_{\lambda,w}; s_i \xx, y; \beta) \bigr)}{x_i - x_{i+1}}.
\]
This is equal to the functional equation given by $\overline{\varpi}_i$.
Therefore, we have the following.

\begin{theorem}[{\cite[Thm.~3.4]{BSW20}}]
We have
\[
Z(\overline{\states}_{\lambda, w_0 w}; \xx, y; \beta) = \overline{L}_{w\lambda}(\xx; \beta, y) := \overline{\varpi}_w \prod_{i=1}^{\ell} (x_i \oplus y)^{\lambda_i}.
\]
\end{theorem}

We call the polynomial $\overline{L}_{w\lambda}(\xx; \beta, y)$ the \defn{extended Lascoux atom} as it equals the Lascoux atom~\cite{Monical16} when $y = 0$.

An important point is if the values of $\yy$ are distinct, then the $R$-matrix depends on the particular values of $\yy$.
In particular, if we consider the $i$-th column as having a fixed value $y_i$, then the weights of the $R$-matrix in Figure~\ref{fig:colored_R_matrix} become
\[
\begin{array}{c@{\hspace{25pt}}c@{\hspace{25pt}}c@{\hspace{25pt}}c}
(1 + \beta x_i) (x_j \oplus y), & (1+ \beta x_i) (x_j \oplus y), & (x_j - x_i) (x_j \oplus y), & (1 + \beta x_j) (x_j \oplus y), \\[5pt]
(1 + \beta x_j) (x_i \oplus y), & (1 + \beta x_i) (x_j \oplus y), & x_j - x_i, & (1 + \beta x_i) (x_j \oplus y).
\end{array}
\]
Therefore, the standard train argument does not hold as the $R$-matrix for passing from the $i$-th column to the $(i+1)$-th column are different.
However, because all of the values of $\yy = y$ are equal, the value of $y$ is no longer an \emph{input} for the $R$-matrix; instead we can think of the $R$-matrix as given by the constant value of $y$ (just like for $\beta$).

\begin{figure}
\[
\begin{array}{c@{\;\;}c@{\;\;}c@{\;\;}c@{\;\;}c@{\;\;}c@{\;\;}c}
\toprule
  \tt{u}_1&\tt{u}'_2&\tt{u}^{\dagger}_2&\tt{u}^{\circ}_2&\tt{v}_2&\tt{w}_1&\tt{w}_2\\
\midrule
\begin{tikzpicture}
\coordinate (a) at (-.75, 0);
\coordinate (b) at (0, .75);
\coordinate (c) at (.75, 0);
\coordinate (d) at (0, -.75);
\coordinate (aa) at (-.75,.5);
\coordinate (cc) at (.75,.5);
\draw (a)--(0,0);
\draw (b)--(0,0);
\draw (c)--(0,0);
\draw (d)--(0,0);
\draw[fill=white] (a) circle (.25);
\draw[fill=white] (b) circle (.25);
\draw[fill=white] (c) circle (.25);
\draw[fill=white] (d) circle (.25);
\node at (0,1) { };
\node at (a) {$0$};
\node at (b) {$0$};
\node at (c) {$0$};
\node at (d) {$0$};
\path[fill=white] (0,0) circle (.2);
\node at (0,0) {$x$};
\end{tikzpicture}
%%%%%%%
& \begin{tikzpicture}
\coordinate (a) at (-.75, 0);
\coordinate (b) at (0, .75);
\coordinate (c) at (.75, 0);
\coordinate (d) at (0, -.75);
\coordinate (aa) at (-.75,.5);
\coordinate (cc) at (.75,.5);
\draw[line width=0.5mm, red] (c)--(0,0);
\draw[line width=0.5mm, red] (b)--(0,0);
\draw[line width=0.5mm, blue] (a)--(0,0);
\draw[line width=0.5mm, blue] (d)--(0,0);
\draw[line width=0.5mm, red,fill=white] (c) circle (.25);
\draw[line width=0.5mm, red,fill=white] (b) circle (.25);
\draw[line width=0.5mm, blue, fill=white] (a) circle (.25);
\draw[line width=0.5mm, blue, fill=white] (d) circle (.25);
\node at (0,1) { };
\node at (c) {$c$};
\node at (b) {$c$};
\node at (a) {$c'$};
\node at (d) {$c'$};
\path[fill=white] (0,0) circle (.2);
\node at (0,0) {$x$};
\end{tikzpicture}
%%%%%%%
& \begin{tikzpicture}
\coordinate (a) at (-.75, 0);
\coordinate (b) at (0, .75);
\coordinate (c) at (.75, 0);
\coordinate (d) at (0, -.75);
\coordinate (aa) at (-.75,.5);
\coordinate (cc) at (.75,.5);
\draw[line width=0.5mm, blue] (a)--(0,0);
\draw[line width=0.6mm, red] (b)--(0,0);
\draw[line width=0.5mm, blue] (c)--(0,0);
\draw[line width=0.6mm, red] (d)--(0,0);
\draw[line width=0.5mm, blue,fill=white] (a) circle (.25);
\draw[line width=0.5mm, red,fill=white] (b) circle (.25);
\draw[line width=0.5mm, blue, fill=white] (c) circle (.25);
\draw[line width=0.5mm, red, fill=white] (d) circle (.25);
\node at (0,1) { };
\node at (a) {$c'$};
\node at (b) {$c$};
\node at (c) {$c'$};
\node at (d) {$c$};
\path[fill=white] (0,0) circle (.2);
\node at (0,0) {$x$};
\end{tikzpicture}
%%%%%%%
& \begin{tikzpicture}
\coordinate (a) at (-.75, 0);
\coordinate (b) at (0, .75);
\coordinate (c) at (.75, 0);
\coordinate (d) at (0, -.75);
\coordinate (aa) at (-.75,.5);
\coordinate (cc) at (.75,.5);
\draw[line width=0.5mm, UQgold] (a)--(0,0);
\draw[line width=0.6mm, UQgold] (b)--(0,0);
\draw[line width=0.5mm, UQgold] (c)--(0,0);
\draw[line width=0.6mm, UQgold] (d)--(0,0);
\draw[line width=0.5mm, UQgold,fill=white] (a) circle (.25);
\draw[line width=0.5mm, UQgold,fill=white] (b) circle (.25);
\draw[line width=0.5mm, UQgold, fill=white] (c) circle (.25);
\draw[line width=0.5mm, UQgold, fill=white] (d) circle (.25);
\node at (0,1) { };
\node at (a) {$d$};
\node at (b) {$d$};
\node at (c) {$d$};
\node at (d) {$d$};
\path[fill=white] (0,0) circle (.2);
\node at (0,0) {$x$};
\end{tikzpicture}
%%%%%%%
& \begin{tikzpicture}
\coordinate (a) at (-.75, 0);
\coordinate (b) at (0, .75);
\coordinate (c) at (.75, 0);
\coordinate (d) at (0, -.75);
\coordinate (aa) at (-.75,.5);
\coordinate (cc) at (.75,.5);
\draw[line width=0.5mm, UQgold] (a)--(0,0);
\draw(b)--(0,0);
\draw[line width=0.5mm, UQgold] (c)--(0,0);
\draw (d)--(0,0);
\draw[line width=0.5mm,UQgold,fill=white] (a) circle (.25);
\draw[fill=white] (b) circle (.25);
\draw[line width=0.5mm,UQgold,fill=white] (c) circle (.25);
\draw[fill=white] (d) circle (.25);
\node at (0,1) { };
\node at (a) {$d$};
\node at (b) {$0$};
\node at (c) {$d$};
\node at (d) {$0$};
\path[fill=white] (0,0) circle (.2);
\node at (0,0) {$x$};
\end{tikzpicture}
%%%%%%%
& \begin{tikzpicture}
\coordinate (a) at (-.75, 0);
\coordinate (b) at (0, .75);
\coordinate (c) at (.75, 0);
\coordinate (d) at (0, -.75);
\coordinate (aa) at (-.75,.5);
\coordinate (cc) at (.75,.5);
\draw[line width=0.5mm, UQgold] (c)--(0,0);
\draw[line width=0.6mm, UQgold] (b)--(0,0);
\draw (a)--(0,0);
\draw (d)--(0,0);
\draw[line width=0.5mm,UQgold,fill=white] (c) circle (.25);
\draw[line width=0.5mm,UQgold,fill=white] (b) circle (.25);
\draw[fill=white] (a) circle (.25);
\draw[fill=white] (d) circle (.25);
\node at (0,1) { };
\node at (c) {$d$};
\node at (b) {$d$};
\node at (a) {$0$};
\node at (d) {$0$};
\path[fill=white] (0,0) circle (.2);
\node at (0,0) {$x$};
\end{tikzpicture}
%%%%%%%
& \begin{tikzpicture}
\coordinate (a) at (-.75, 0);
\coordinate (b) at (0, .75);
\coordinate (c) at (.75, 0);
\coordinate (d) at (0, -.75);
\coordinate (aa) at (-.75,.5);
\coordinate (cc) at (.75,.5);
\draw (c)--(0,0);
\draw (b)--(0,0);
\draw[line width=0.5mm, UQgold] (a)--(0,0);
\draw[line width=0.5mm, UQgold] (d)--(0,0);
\draw[fill=white] (c) circle (.25);
\draw[fill=white] (b) circle (.25);
\draw[line width=0.5mm,UQgold, fill=white] (a) circle (.25);
\draw[line width=0.5mm,UQgold, fill=white] (d) circle (.25);
\node at (0,1) { };
\node at (c) {$0$};
\node at (b) {$0$};
\node at (a) {$d$};
\node at (d) {$d$};
\path[fill=white] (0,0) circle (.2);
\node at (0,0) {$x$};
\end{tikzpicture}
%%%%%%%
\\ 
\midrule
1 & 1 &1 & 1 & x \oplus y & 1 & 1+\beta (x \oplus y) \\
\bottomrule
\end{array}
\]
\caption{The colored Boltzmann weights for the extended Lascoux polynomial with ${\color{red} c} > {\color{blue} c'}$ and ${\color{UQgold} d}$ being any color.}
\label{fig:colored_weights_Lpoly}
\end{figure}

Furthermore, as in~\cite{BSW20}, we also get an integrable model using the $L$-matrix given by Table~\ref{fig:colored_weights_Lpoly}, which differs by replacing the vertex $\tt{b}_1$ with the colors swapped, which we call $\tt{b}'_1$.
We denote the admissible states of this model by $\states_{\lambda, w_0 w}$ with the same boundary conditions as before.
Therefore, we have the following.

\begin{theorem}[{\cite[Thm.~3.4]{BSW20}}]
We have
\[
Z(\states_{\lambda, w_0 w}; \xx, y; \beta) = L_{w\lambda}(\xx; \beta, y) := \varpi_w \prod_{i=1}^{\ell} (x_i \oplus y)^{\lambda_i}.
\]
\end{theorem}

We call $L_{w\lambda}(\xx; \beta, y)$ the \defn{extended Lascoux polynomial}.

If we consider the uncolored model (equivalently having precisely one color that we call $1$), then the corresponding $R$-matrix is independent of $y$~\cite{GK17,Motegi20} (see Remark~\ref{rem:specialized_models}).
Note that from the above discussion there is no natural colored version of this model.
We denote by $\states_{\lambda}$ the uncolored model with the right boundary being $1$, left and bottom boundaries being $0$, and top boundary being the $01$-sequence of $\lambda$ in reverse.
Using the natural bijection between with Gelfand--Tsetlin patterns from~\cite[Sec.~4.2]{MPS18}, we can see the partition function is precisely the factorial Grothendieck polynomial.

\begin{theorem}[{\cite{GK17,MS13,Motegi20,WZJ19}}]
\label{thm:uncolored_factorial}
We have
\[
Z(\states_{\lambda}; \xx, \yy; \beta) = \G_{\lambda}(\xx | \yy;  \beta).
\]
\end{theorem}

%===============================================================================
\section{Double Grothendieck polynomial colored model}
\label{sec:double_Grothendieck_model}

In this section, we give our main result: an integrable colored model whose partition function is a double Grothendieck polynomial.
We then explain how this model corresponds to the bumpless pipe dream formula from~\cite{Weigandt20}.

%%%%%%%%%%%%%%%%%%%%
\subsection{The colored model}
\label{sec:main_colored_model}

We begin with a remark about how we are drawing our $L$-matrices relate to their representation-theoretic interpretation.

\begin{remark}
For the previous colored model, the $L$-matrix given by $L(a \otimes b) = b' \otimes a'$ by the picture on the left, whereas in this section we use the picture on the right
\[
\begin{tikzpicture}[scale=0.8]
\draw[->] (0,-1) node[anchor=north] {$b$} -- (0,1) node[anchor=south] {$b'$};
\draw[->] (1,0) node[anchor=west] {$a$} -- (-1,0) node[anchor=east] {$a'$};
\begin{scope}[xshift=7cm]
\draw[->] (0,-1) node[anchor=north] {$b$} -- (0,1) node[anchor=south] {$b'$};
\draw[->] (-1,0) node[anchor=east] {$a$} -- (1,0) node[anchor=west] {$a'$};
\end{scope}
\end{tikzpicture}
\]
\end{remark}

\begin{figure}
\[
\begin{array}{c@{\;\;}c@{\;\;}c@{\;\;}c@{\;\;}c@{\;\;}c@{\;\;}c}
\toprule
  \tt{a}_1&\tt{a}_2&\tt{a}^{\dagger}_2&\tt{b}_1&\tt{b}_2&\tt{c}_1&\tt{c}_2\\
\midrule
\begin{tikzpicture}
\coordinate (a) at (-.75, 0);
\coordinate (b) at (0, .75);
\coordinate (c) at (.75, 0);
\coordinate (d) at (0, -.75);
\coordinate (aa) at (-.75,.5);
\coordinate (cc) at (.75,.5);
\draw (a)--(0,0);
\draw (b)--(0,0);
\draw (c)--(0,0);
\draw (d)--(0,0);
\draw[fill=white] (a) circle (.25);
\draw[fill=white] (b) circle (.25);
\draw[fill=white] (c) circle (.25);
\draw[fill=white] (d) circle (.25);
\node at (0,1) { };
\node at (a) {$0$};
\node at (b) {$0$};
\node at (c) {$0$};
\node at (d) {$0$};
\path[fill=white] (0,0) circle (.2);
\node at (0,0) {$x$};
\end{tikzpicture}
%%%%%%%
& \begin{tikzpicture}
\coordinate (a) at (-.75, 0);
\coordinate (b) at (0, .75);
\coordinate (c) at (.75, 0);
\coordinate (d) at (0, -.75);
\coordinate (aa) at (-.75,.5);
\coordinate (cc) at (.75,.5);
\draw[line width=0.5mm, blue] (a)--(0,0);
\draw[line width=0.5mm, blue] (b)--(0,0);
\draw[line width=0.5mm, red] (c)--(0,0);
\draw[line width=0.5mm, red] (d)--(0,0);
\draw[line width=0.5mm, blue, fill=white] (a) circle (.25);
\draw[line width=0.5mm, blue, fill=white] (b) circle (.25);
\draw[line width=0.5mm, red, fill=white] (c) circle (.25);
\draw[line width=0.5mm, red, fill=white] (d) circle (.25);
\node at (0,1) { };
\node at (a) {$c'$};
\node at (b) {$c'$};
\node at (c) {$c$};
\node at (d) {$c$};
\path[fill=white] (0,0) circle (.2);
\node at (0,0) {$x$};
\end{tikzpicture}
%%%%%%%
& \begin{tikzpicture}
\coordinate (a) at (-.75, 0);
\coordinate (b) at (0, .75);
\coordinate (c) at (.75, 0);
\coordinate (d) at (0, -.75);
\coordinate (aa) at (-.75,.5);
\coordinate (cc) at (.75,.5);
\draw[line width=0.5mm, red] (a)--(0,0);
\draw[line width=0.6mm, blue] (b)--(0,0);
\draw[line width=0.5mm, red] (c)--(0,0);
\draw[line width=0.6mm, blue] (d)--(0,0);
\draw[line width=0.5mm, red, fill=white] (a) circle (.25);
\draw[line width=0.5mm, blue, fill=white] (b) circle (.25);
\draw[line width=0.5mm, red, fill=white] (c) circle (.25);
\draw[line width=0.5mm, blue, fill=white] (d) circle (.25);
\node at (0,1) { };
\node at (a) {$c$};
\node at (b) {$c'$};
\node at (c) {$c$};
\node at (d) {$c'$};
\path[fill=white] (0,0) circle (.2);
\node at (0,0) {$x$};
\end{tikzpicture}
%%%%%%%
& \begin{tikzpicture}
\coordinate (a) at (-.75, 0);
\coordinate (b) at (0, .75);
\coordinate (c) at (.75, 0);
\coordinate (d) at (0, -.75);
\coordinate (aa) at (-.75,.5);
\coordinate (cc) at (.75,.5);
\draw (a)--(0,0);
\draw[line width=0.6mm, UQgold] (b)--(0,0);
\draw (c)--(0,0);
\draw[line width=0.5mm, UQgold] (d)--(0,0);
\draw[fill=white] (a) circle (.25);
\draw[line width=0.5mm,UQgold,fill=white] (b) circle (.25);
\draw[fill=white] (c) circle (.25);
\draw[line width=0.5mm,UQgold,fill=white] (d) circle (.25);
\node at (0,1) { };
\node at (a) {$0$};
\node at (b) {$d$};
\node at (c) {$0$};
\node at (d) {$d$};
\path[fill=white] (0,0) circle (.2);
\node at (0,0) {$x$};
\end{tikzpicture}
%%%%%%%
& \begin{tikzpicture}
\coordinate (a) at (-.75, 0);
\coordinate (b) at (0, .75);
\coordinate (c) at (.75, 0);
\coordinate (d) at (0, -.75);
\coordinate (aa) at (-.75,.5);
\coordinate (cc) at (.75,.5);
\draw[line width=0.5mm, UQgold] (a)--(0,0);
\draw(b)--(0,0);
\draw[line width=0.5mm, UQgold] (c)--(0,0);
\draw (d)--(0,0);
\draw[line width=0.5mm,UQgold,fill=white] (a) circle (.25);
\draw[fill=white] (b) circle (.25);
\draw[line width=0.5mm,UQgold,fill=white] (c) circle (.25);
\draw[fill=white] (d) circle (.25);
\node at (0,1) { };
\node at (a) {$d$};
\node at (b) {$0$};
\node at (c) {$d$};
\node at (d) {$0$};
\path[fill=white] (0,0) circle (.2);
\node at (0,0) {$x$};
\end{tikzpicture}
%%%%%%%
& \begin{tikzpicture}
\coordinate (a) at (-.75, 0);
\coordinate (b) at (0, .75);
\coordinate (c) at (.75, 0);
\coordinate (d) at (0, -.75);
\coordinate (aa) at (-.75,.5);
\coordinate (cc) at (.75,.5);
\draw[line width=0.5mm, UQgold] (a)--(0,0);
\draw[line width=0.6mm, UQgold] (b)--(0,0);
\draw (c)--(0,0);
\draw (d)--(0,0);
\draw[line width=0.5mm,UQgold,fill=white] (a) circle (.25);
\draw[line width=0.5mm,UQgold,fill=white] (b) circle (.25);
\draw[fill=white] (c) circle (.25);
\draw[fill=white] (d) circle (.25);
\node at (0,1) { };
\node at (a) {$d$};
\node at (b) {$d$};
\node at (c) {$0$};
\node at (d) {$0$};
\path[fill=white] (0,0) circle (.2);
\node at (0,0) {$x$};
\end{tikzpicture}
%%%%%%%
& \begin{tikzpicture}
\coordinate (a) at (-.75, 0);
\coordinate (b) at (0, .75);
\coordinate (c) at (.75, 0);
\coordinate (d) at (0, -.75);
\coordinate (aa) at (-.75,.5);
\coordinate (cc) at (.75,.5);
\draw (a)--(0,0);
\draw (b)--(0,0);
\draw[line width=0.5mm, UQgold] (c)--(0,0);
\draw[line width=0.5mm, UQgold] (d)--(0,0);
\draw[fill=white] (a) circle (.25);
\draw[fill=white] (b) circle (.25);
\draw[line width=0.5mm,UQgold, fill=white] (c) circle (.25);
\draw[line width=0.5mm,UQgold, fill=white] (d) circle (.25);
\node at (0,1) { };
\node at (a) {$0$};
\node at (b) {$0$};
\node at (c) {$d$};
\node at (d) {$d$};
\path[fill=white] (0,0) circle (.2);
\node at (0,0) {$x$};
\end{tikzpicture}
%%%%%%%
\\ 
\midrule
\beta (x \oplus y) & 1 &1  & 1 & 1 & 1 + \beta (x \oplus y) & 1 \\
\bottomrule
\end{array}
\]
\caption{The colored Boltzmann weights with ${\color{red} c} > {\color{blue} c'}$ and ${\color{UQgold} d}$ being any color for the double Grothendieck model.}
\label{fig:bumpless_weights}
\end{figure}

We construct a vertex model $\Gstates_w$, which we call the \defn{double Grothendieck model} on an $n \times n$ grid using the $L$-matrix given by Figure~\ref{fig:bumpless_weights} with the following boundary conditions.
We consider colors $\cc = (c_1 > c_2 > \cdots > c_n)$.
We let the left and top boundary conditions be $0$, the bottom boundary being $\cc$ (from left-to-right), and the right boundary being $w\cc$ (from top-to-bottom).

\begin{proposition}
\label{prop:double_Grothendieck_integrable}
The colored vertex model $\Gstates_w$ is integrable with $R$-matrix given by Figure~\ref{fig:bumpless_R_matrix}.
\end{proposition}

See Appendix~\ref{sec:sage_code} for the \textsc{SageMath}~\cite{sage} code for showing Proposition~\ref{prop:double_Grothendieck_integrable}.

\begin{figure}%[h]
\[
\begin{array}{c@{\hspace{30pt}}c@{\hspace{30pt}}c@{\hspace{30pt}}c}
\toprule
%  ++++ : z2
\begin{tikzpicture}[scale=0.7]
\draw (0,0) to [out = 0, in = 180] (2,2);
\draw (0,2) to [out = 0, in = 180] (2,0);
\draw[fill=white] (0,0) circle (.35);
\draw[fill=white] (0,2) circle (.35);
\draw[fill=white] (2,0) circle (.35);
\draw[fill=white] (2,2) circle (.35);
\node at (0,0) {$0$};
\node at (0,2) {$0$};
\node at (2,2) {$0$};
\node at (2,0) {$0$};
\path[fill=white] (1,1) circle (.3);
\node at (1,1) {$x_i,x_j$};
\end{tikzpicture}&
%  +RR+ : z2
\begin{tikzpicture}[scale=0.7]
\draw (0,0) to [out = 0, in = 180] (2,2);
\draw (0,2) to [out = 0, in = 180] (2,0);
\draw[fill=white] (0,0) circle (.35);
\draw[line width=0.5mm, UQgold, fill=white] (0,2) circle (.35);
\draw[line width=0.5mm, UQgold, fill=white] (2,2) circle (.35);
\draw[fill=white] (2,0) circle (.35);
\node at (0,0) {$0$};
\node at (0,2) {$d$};
\node at (2,2) {$d$};
\node at (2,0) {$0$};
\path[fill=white] (1,1) circle (.3);
\node at (1,1) {$x_i,x_j$};
\end{tikzpicture}&
%  R++R : z1
\begin{tikzpicture}[scale=0.7]
\draw (0,0) to [out = 0, in = 180] (2,2);
\draw (0,2) to [out = 0, in = 180] (2,0);
\draw[fill=white] (0,0) circle (.35);
\draw[line width=0.5mm, UQgold, fill=white] (0,2) circle (.35);
\draw[fill=white] (2,2) circle (.35);
\draw[line width=0.5mm, UQgold, fill=white] (2,0) circle (.35);
\node at (0,0) {$0$};
\node at (0,2) {$d$};
\node at (2,2) {$0$};
\node at (2,0) {$d$};
\path[fill=white] (1,1) circle (.3);
\node at (1,1) {$x_i,x_j$};
\end{tikzpicture}&
%  R++R : z2
\begin{tikzpicture}[scale=0.7]
\draw (0,0) to [out = 0, in = 180] (2,2);
\draw (0,2) to [out = 0, in = 180] (2,0);
\draw[line width=0.5mm, UQgold, fill=white] (0,0) circle (.35);
\draw[fill=white] (0,2) circle (.35);
\draw[fill=white] (2,2) circle (.35);
\draw[line width=0.5mm, UQgold, fill=white] (2,0) circle (.35);
\node at (0,0) {$d$};
\node at (0,2) {$0$};
\node at (2,2) {$0$};
\node at (2,0) {$d$};
\path[fill=white] (1,1) circle (.3);
\node at (1,1) {$x_i,x_j$};
\end{tikzpicture}\\
   \midrule
 1 + \beta x_i & 1+ \beta x_i & \beta (x_j - x_i) & 1 + \beta x_j \\
   \midrule
%  BRRB : z1
\begin{tikzpicture}[scale=0.7]
\draw (0,0) to [out = 0, in = 180] (2,2);
\draw (0,2) to [out = 0, in = 180] (2,0);
\draw[line width=0.5mm, blue, fill=white] (0,0) circle (.35);
\draw[line width=0.5mm, red, fill=white] (0,2) circle (.35);
\draw[line width=0.5mm, red, fill=white] (2,2) circle (.35);
\draw[line width=0.5mm, blue, fill=white] (2,0) circle (.35);
\node at (0,0) {$c'$};
\node at (0,2) {$c$};
\node at (2,2) {$c$};
\node at (2,0) {$c'$};
\path[fill=white] (1,1) circle (.3);
\node at (1,1) {$x_i,x_j$};
\end{tikzpicture}&
%  RBBR : z2
\begin{tikzpicture}[scale=0.7]
\draw (0,0) to [out = 0, in = 180] (2,2);
\draw (0,2) to [out = 0, in = 180] (2,0);
\draw[line width=0.5mm, red, fill=white] (0,0) circle (.35);
\draw[line width=0.5mm, blue, fill=white] (0,2) circle (.35);
\draw[line width=0.5mm, blue, fill=white] (2,2) circle (.35);
\draw[line width=0.5mm, red, fill=white] (2,0) circle (.35);
\node at (0,0) {$c$};
\node at (0,2) {$c'$};
\node at (2,2) {$c'$};
\node at (2,0) {$c$};
\path[fill=white] (1,1) circle (.3);
\node at (1,1) {$x_i,x_j$};
\end{tikzpicture}&
%  RBRB : z1 - z2
\begin{tikzpicture}[scale=0.7]
\draw (0,0) to [out = 0, in = 180] (2,2);
\draw (0,2) to [out = 0, in = 180] (2,0);
\draw[line width=0.5mm, blue, fill=white] (0,0) circle (.35);
\draw[line width=0.5mm, red, fill=white] (0,2) circle (.35);
\draw[line width=0.5mm, blue, fill=white] (2,2) circle (.35);
\draw[line width=0.5mm, red, fill=white] (2,0) circle (.35);
\node at (0,0) {$c'$};
\node at (0,2) {$c$};
\node at (2,2) {$c'$};
\node at (2,0) {$c$};
\path[fill=white] (1,1) circle (.3);
\node at (1,1) {$x_i,x_j$};
\end{tikzpicture}&
%  RRRR : z2
\begin{tikzpicture}[scale=0.7]
\draw (0,0) to [out = 0, in = 180] (2,2);
\draw (0,2) to [out = 0, in = 180] (2,0);
\draw[line width=0.5mm, UQgold, fill=white] (0,0) circle (.35);
\draw[line width=0.5mm, UQgold, fill=white] (0,2) circle (.35);
\draw[line width=0.5mm, UQgold, fill=white] (2,2) circle (.35);
\draw[line width=0.5mm, UQgold, fill=white] (2,0) circle (.35);
\node at (0,0) {$d$};
\node at (0,2) {$d$};
\node at (2,2) {$d$};
\node at (2,0) {$d$};
\path[fill=white] (1,1) circle (.3);
\node at (1,1) {$x_i,x_j$};
\end{tikzpicture}\\
   \midrule
1 + \beta x_i & 1 + \beta x_j & \beta (x_j - x_i) & 1 + \beta x_i \\
   \bottomrule
\end{array}
\]
\caption{The colored $R$-matrix with ${\color{red} c} > {\color{blue} c'}$ and ${\color{UQgold} d}$ being any color for the double Grothendieck model. Note that the weights are not symmetric with respect to color.}
\label{fig:bumpless_R_matrix}
\end{figure}
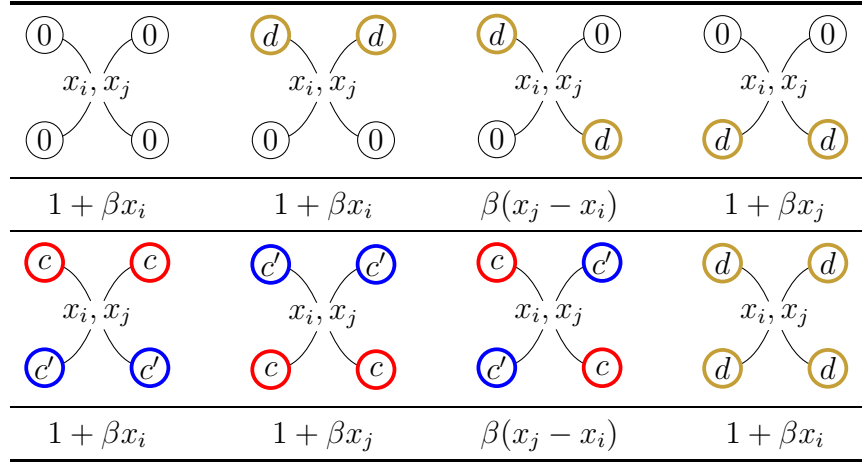
 
Following~\cite{BBBG19,BSW20}, we use the Yang--Baxter equation and the standard train argument to construct a functional equation for the partition function of our model $\Gstates_w$.
The proof is entirely standard and is the same as the one given in, \textit{e.g.},~\cite[Lemma~3.3]{BSW20}.
See Figure~\ref{fig:train_argument} for a pictorial description of the proof.

\begin{lemma}
\label{lemma:recurrence_eq}
Let $w \in \sym{n}$, and consider any $s_i$ such that $w s_i > w$.
Then we have
\[
\beta Z(\Gstates_w; \xx, \yy; \beta) = \frac{(1 + \beta x_{i+1} ) \cdot Z(\Gstates_{w s_i}; \xx, \yy; \beta) - (1 + \beta x_i) \cdot Z(\Gstates_{w s_i}; s_i \xx, \yy; \beta)}{x_i - x_{i+1}}.
\]
\end{lemma}

Now we show that the partition function satisfies the same functional equation and initial condition as the double Grothendieck polynomials to obtain our main result.

\begin{theorem}
\label{thm:5vertex_atoms}
We have
\[
Z(\Gstates_w; \xx, \yy; \beta) = \beta^{\ell(w)} \G_w(\xx, \yy; \beta).
\]
\end{theorem}

\begin{proof}
A direct computation using the definition of $\Kdd_i$ yields
\[
\G_w(\xx, \yy; \beta) = \Kdd_i \G_{w s_i}(\xx, \yy; \beta)
= \frac{(1  + \beta x_{i+1} ) \cdot \G_{w s_i}(\xx, \yy; \beta) -  (1 + \beta x_i) \cdot \G_{w s_i}(s_i \xx, \yy; \beta)}{x_i - x_{i+1}}.
\]
It is straightforward to see that
\[
Z(\Gstates_{w_0}; \xx, \yy; \beta) = \beta^{n(n-1)/2} \left(\prod_{i + j \leq n} x_i \oplus y_j \right) = \beta^{\ell(w_0)} \G_{w_0}(\xx, \yy; \beta)
\]
(see Figure~\ref{fig:ground_state} for an example).
Therefore, the claim follows by induction and Lemma~\ref{lemma:recurrence_eq}.
\end{proof}

\begin{figure}
  \begin{tikzpicture}[scale=0.75, font=\small,baseline=0]
    % Draw the grid and weights
    \foreach \i in {1,3,5,7} {
        \draw (\i,0)--(\i,8);
        \draw (0,\i)--(8,\i);
    }
    \draw[line width=0.5mm,red] (1,0) -- (1,1) -- (8,1);
    \draw[line width=0.5mm,blue] (3,0) -- (3,3) -- (8,3);
    \draw[line width=0.5mm,darkgreen] (5,0) -- (5,5) -- (8,5);
    \draw[line width=0.5mm,UQpurple] (7,0) -- (7,7) -- (8,7);
    \foreach \x in {0,2,4,6,8} {
        \foreach \y in {1,3,5,7} {
            \draw[fill=white] (\x,\y) circle (.35); \node at (\x,\y) {$0$};
        }
    }
    \foreach \i in {1,3,5,7} {
        \foreach \y in {0,2,4,6,8} {
            \draw[fill=white] (\i,\y) circle (.35); \node at (\i,\y) {$0$};
        }
        \node at (\i,8) {$0$};
        \node at (0,\i) {$0$};
    }
    
    % c_1 path
    \foreach \y in {0} {
        \draw[line width=0.5mm,red,fill=white] (1,\y) circle (.35); \node at (1,\y) {$c_1$};
    }
    \foreach \x in {2,4,6,8}{
        \draw[line width=0.5mm,red,fill=white] (\x,1) circle (.35); \node at (\x,1) {$c_1$};
    }
    % c_2 path
    \foreach \y in {0,2} {
        \draw[line width=0.5mm,blue,fill=white] (3,\y) circle (.35); \node at (3,\y) {$c_2$};
    }
    \foreach \x in {4,6,8}{
        \draw[line width=0.5mm,blue,fill=white] (\x,3) circle (.35); \node at (\x,3) {$c_2$};
    }
    % c_3 path
    \foreach \y in {0,2,4} {
        \draw[line width=0.5mm,darkgreen,fill=white] (5,\y) circle (.35); \node at (5,\y) {$c_3$};
    }
    \foreach \x in {6,8}{
        \draw[line width=0.5mm,darkgreen,fill=white] (\x,5) circle (.35); \node at (\x,5) {$c_3$};
    }
    % c_4 path
    \foreach \y in {0,2,4,6} {
        \draw[line width=0.5mm,UQpurple,fill=white] (7,\y) circle (.35); \node at (7,\y) {$c_4$};
    }
    \foreach \x in {8}{
        \draw[line width=0.5mm,UQpurple,fill=white] (\x,7) circle (.35); \node at (\x,7) {$c_4$};
    }

    % The extra labels
    \foreach \i in {1,3,5,7} {
        \foreach \y in {1,2,3,4} {
            \path[fill=white] (\i,9-2*\y) circle (.3);
            \node at (\i,9-2*\y) {$x_{\y}$};
        }
    }
    \foreach \i in {1,2,3,4}{
        \node at (2*\i-1,8.8) {$ \i$};
        \node at(-.75,9-2*\i) {$ \i$};
    }
  \end{tikzpicture}
  \caption{The unique state for the colored model $\Gstates_{w_0}$ with $n = 4$ and colors ${\color{red} c_1} > {\color{blue} c_2} > {\color{darkgreen} c_3} > {\color{UQpurple} c_4}$.}
  \label{fig:ground_state}
\end{figure}

\begin{figure}
  \begin{tikzpicture}[scale=0.75, font=\small,baseline=0]
    % Draw the grid and weights
    \foreach \i in {1,3,5,7} {
        \draw (\i,0)--(\i,8);
        \draw (0,\i)--(8,\i);
    }
    \draw[line width=0.5mm,red] (1,0) -- (1,3) -- (5,3) -- (5,5) -- (8,5);
    \draw[line width=0.5mm,blue] (3,0) -- (3,5) -- (5,5) -- (5,7) -- (8,7);
    \draw[line width=0.5mm,darkgreen] (5,0) -- (5,1) -- (8,1);
    \draw[line width=0.5mm,UQpurple] (7,0) -- (7,3) -- (8,3);
    \foreach \x in {0,2,4,6,8} {
        \foreach \y in {1,3,5,7} {
            \draw[fill=white] (\x,\y) circle (.35); \node at (\x,\y) {$0$};
        }
    }
    \foreach \i in {1,3,5,7} {
        \foreach \y in {0,2,4,6,8} {
            \draw[fill=white] (\i,\y) circle (.35); \node at (\i,\y) {$0$};
        }
        \node at (\i,8) {$0$};
        \node at (0,\i) {$0$};
    }
    
    % c_1 path
    \foreach \x/\y in {1/0,1/2,2/3,4/3,5/4,6/5,8/5} {
        \draw[line width=0.5mm,red,fill=white] (\x,\y) circle (.35); \node at (\x,\y) {$c_1$};
    }
    % c_2 path
    \foreach \x/\y in {3/0,3/2,3/4,4/5,5/6,6/7,8/7} {
        \draw[line width=0.5mm,blue,fill=white] (\x,\y) circle (.35); \node at (\x,\y) {$c_2$};
    }
    % c_3 path
    \foreach \y in {0} {
        \draw[line width=0.5mm,darkgreen,fill=white] (5,\y) circle (.35); \node at (5,\y) {$c_3$};
    }
    \foreach \x in {6,8}{
        \draw[line width=0.5mm,darkgreen,fill=white] (\x,1) circle (.35); \node at (\x,1) {$c_3$};
    }
    % c_4 path
    \foreach \y in {0,2} {
        \draw[line width=0.5mm,UQpurple,fill=white] (7,\y) circle (.35); \node at (7,\y) {$c_4$};
    }
    \foreach \x in {8}{
        \draw[line width=0.5mm,UQpurple,fill=white] (\x,3) circle (.35); \node at (\x,3) {$c_4$};
    }

    % The extra labels
    \foreach \i in {1,3,5,7} {
        \foreach \y in {1,2,3,4} {
            \path[fill=white] (\i,9-2*\y) circle (.3);
            \node at (\i,9-2*\y) {$x_{\y}$};
        }
    }
    \foreach \i in {1,2,3,4}{
        \node at (2*\i-1,8.8) {$ \i$};
        \node at(-.75,9-2*\i) {$ \i$};
    }
  \end{tikzpicture}
  \qquad
  \begin{tikzpicture}[scale=0.75, font=\small,baseline=0]
    % Draw the grid and weights
    \foreach \i in {1,3,5,7} {
        \draw (\i,0)--(\i,8);
        \draw (0,\i)--(8,\i);
    }
    \draw[line width=0.5mm,red] (1,0) -- (1,5) -- (5,5) -- (5,7) -- (8,7);
    \draw[line width=0.5mm,blue] (3,0) -- (3,1) -- (8,1);
    \draw[line width=0.5mm,darkgreen] (5,0) -- (5,3) -- (8,3);
    \draw[line width=0.5mm,UQpurple] (7,0) -- (7,5) -- (8,5);
    \foreach \x in {0,2,4,6,8} {
        \foreach \y in {1,3,5,7} {
            \draw[fill=white] (\x,\y) circle (.35); \node at (\x,\y) {$0$};
        }
    }
    \foreach \i in {1,3,5,7} {
        \foreach \y in {0,2,4,6,8} {
            \draw[fill=white] (\i,\y) circle (.35); \node at (\i,\y) {$0$};
        }
        \node at (\i,8) {$0$};
        \node at (0,\i) {$0$};
    }
    
    % c_1 path
    \foreach \x/\y in {1/0,1/2,1/4,2/5,4/5,5/6,6/7,8/7} {
        \draw[line width=0.5mm,red,fill=white] (\x,\y) circle (.35); \node at (\x,\y) {$c_1$};
    }
    % c_2 path
    \foreach \y in {0} {
        \draw[line width=0.5mm,blue,fill=white] (3,\y) circle (.35); \node at (3,\y) {$c_2$};
    }
    \foreach \x in {4,6,8}{
        \draw[line width=0.5mm,blue,fill=white] (\x,1) circle (.35); \node at (\x,1) {$c_2$};
    }
    % c_3 path
    \foreach \y in {0,2} {
        \draw[line width=0.5mm,darkgreen,fill=white] (5,\y) circle (.35); \node at (5,\y) {$c_3$};
    }
    \foreach \x in {6,8}{
        \draw[line width=0.5mm,darkgreen,fill=white] (\x,3) circle (.35); \node at (\x,3) {$c_3$};
    }
    % c_4 path
    \foreach \y in {0,2,4} {
        \draw[line width=0.5mm,UQpurple,fill=white] (7,\y) circle (.35); \node at (7,\y) {$c_4$};
    }
    \foreach \x in {8}{
        \draw[line width=0.5mm,UQpurple,fill=white] (\x,5) circle (.35); \node at (\x,5) {$c_4$};
    }

    % The extra labels
    \foreach \i in {1,3,5,7} {
        \foreach \y in {1,2,3,4} {
            \path[fill=white] (\i,9-2*\y) circle (.3);
            \node at (\i,9-2*\y) {$x_{\y}$};
        }
    }
    \foreach \i in {1,2,3,4}{
        \node at (2*\i-1,8.8) {$ \i$};
        \node at(-.75,9-2*\i) {$ \i$};
    }
  \end{tikzpicture}
  \caption{A state in $\Gstates_{s_1 s_3}$ (left) and $\Gstates_{s_2 s_3 s_2}$ (right) with $n = 4$ and colors ${\color{red} c_1} > {\color{blue} c_2} > {\color{darkgreen} c_3} > {\color{UQpurple} c_4}$.}
  \label{fig:colored_state}
\end{figure}

We can also recover the following well-known symmetry of double Grothendieck polynomials, which can be seen from the formula
\[
\G_w(\xx,\yy; \beta) = \sum_{\substack{uv = w \\ \ell(u) + \ell(v) = \ell(w)}} \G_u(\xx; \beta) \G_v(\yy; \beta)
\]
from~\cite[Thm.~8.1]{FK96} (see also~\cite{FK94}) or from the interpretation of $\G_w(\xx, \yy; \beta)$ as a sum over usual pipe dreams~\cite{BB93,FK94,FK96,KM04}.\footnote{These are also known as RC-graphs.}
The authors thank Anatol Kirillov for noting this.
A colored lattice model proof of this was first given in~\cite[Prop.~9.1]{frozenpipes}.

\begin{corollary}
\label{cor:symmetry}
We have
\[
Z(\states_w; \xx, \yy; \beta) = Z(\states_{w^{-1}}; \yy, \xx; \beta).
\]
Moreover, we have $\G_w(\xx, \yy; \beta) = \G_{w^{-1}}(\yy, \xx; \beta)$.
\end{corollary}

\begin{proof}
This can be seen by reflecting the model across the main diagonal (thus interchanging $x_i \leftrightarrow y_i$).
Note that we are taking the Demazure product in the opposite direction, which is equivalent to reversing the ordering on the colors.
However, this results in an equivalent colored lattice model.
\end{proof}

%%%%%%%%%%%%%%%%%%%%
\subsection{Bumpless pipe dreams}
\label{sec:bumpless_connection}

We now relate the states our model to bumpless pipe dreams given in~\cite{Weigandt20}.
So we obtain a new proof of~\cite[Thm.~1.1]{Weigandt20} that double Grothendieck polynomials are a sum over bumpless pipe dreams by a weight preserving bijection as a corollary of Theorem~\ref{thm:5vertex_atoms}.
In~\cite[Eq.~(3.7)]{Weigandt20}, we see that each uncolored vertex precisely corresponds to one of the tiles that defines a bumpless pipe dream.\footnote{The states of the uncolored model are also equivalent to alternating sign matrices (see~\cite[Sec.~7]{EKLP92}).}
The key $\partial P$ of a bumpless pipe dream $P$ is described using the Demazure product (where we instead consider $s_i^2 = s_i$) of the corresponding product of simple transpositions, which we describe pictorially as follows.
Color each line starting from the bottom as $1, \dotsc, n$ (or $c_1, \dotsc, c_n$) from left-to-right where each color must cross (\textit{i.e}, look like $\tt{a}_2^{\dagger}$ or with the colors swapped or equivalently no vertices look like $\tt{a}_2$ or with its colors swapped).
Moving from the bottom-left to the top-right along diagonals, every time we see a local configuration that looks like the color flipped $\tt{a}_2^{\dagger}$, we replace it with $\tt{a}_2$ (see also~\cite[Eq.~(2.4), Lemma~2.1]{Weigandt20} for another pictorial description).
Thus, for any permutation $w$, the states of our colored vertex model $\Gstates_w$ correspond precisely to those bumpless pipe dreams whose key is $w$ after forgetting about the colors.

%===============================================================================
\section{The semidual vertex model}
\label{sec:dual_model}

In this section, we construct a model using the $L$-matrix for the extended Lascoux polynomial given in Figure~\ref{fig:colored_weights_Lpoly} and relating it to the case when $w$ is a vexillary permutation.
We will also draw the states using tiles describing the connections of the corresponding wiring diagram.
Thus the $L$-matrix given in Figure~\ref{fig:bumpless_weights} become the tiles
\[
\begin{tikzpicture}[scale=0.5,baseline=10]
\draw[very thin,gray] (0,0) rectangle (2,2);
\end{tikzpicture}\,,
\qquad
\begin{tikzpicture}[scale=0.5,baseline=10]
\draw[very thin,gray] (0,0) rectangle (2,2);
\draw[very thick, red] (1,0) -- (1,1) -- (2,1);
\draw[very thick, blue] (0,1) -- (1,1) -- (1,2);
\end{tikzpicture}\,,
\qquad
\begin{tikzpicture}[scale=0.5,baseline=10]
\draw[very thin,gray] (0,0) rectangle (2,2);
\draw[very thick, blue] (1,0) -- (1,2);
\draw[very thick, red] (0,1) -- (2,1);
\end{tikzpicture}\,,
\qquad
\begin{tikzpicture}[scale=0.5,baseline=10]
\draw[very thin,gray] (0,0) rectangle (2,2);
\draw[very thick, UQgold] (1,0) -- (1,2);
\end{tikzpicture}\,,
\qquad
\begin{tikzpicture}[scale=0.5,baseline=10]
\draw[very thin,gray] (0,0) rectangle (2,2);
\draw[very thick, UQgold] (0,1) -- (2,1);
\end{tikzpicture}\,,
\qquad
\begin{tikzpicture}[scale=0.5,baseline=10]
\draw[very thin,gray] (0,0) rectangle (2,2);
\draw[very thick, UQgold] (0,1) -- (1,1) -- (1,2);
\end{tikzpicture}\,,
\qquad
\begin{tikzpicture}[scale=0.5,baseline=10]
\draw[very thin,gray] (0,0) rectangle (2,2);
\draw[very thick, UQgold] (1,0) -- (1,1) -- (2,1);
\end{tikzpicture}\,.
\]

%%%%%%%%%%%%%%%%%%%%
\subsection{The semidual model}

By~\cite[Lemma~7.2]{Weigandt20}, a bumpless pipe dream state has no $\tt{a}_2$ vertex if and only if $w$ is vexillary.
Thus we set the corresponding Boltzmann weight to $0$.
This allows us to forget about the colors when drawing the lattice models, but it does not quite equate the colored model with the uncolored model as we still need to keep track of which strands cross to encode the permutation $w$.
This distinction is important as the model needs to (implicitly) remain colored in order to get the correct partition function corresponding to $w$.
Let $\Gstates$ denote the corresponding uncolored model.

\begin{figure}
\[
\begin{array}{cccccc}
\toprule
  \overline{\tt{a}}_1&\overline{\tt{a}}_2&\overline{\tt{b}}_1&\overline{\tt{b}}_2&\overline{\tt{c}}_1&\overline{\tt{c}}_2\\
\midrule
\begin{tikzpicture}
\coordinate (a) at (-.75, 0);
\coordinate (b) at (0, .75);
\coordinate (c) at (.75, 0);
\coordinate (d) at (0, -.75);
\coordinate (aa) at (-.75,.5);
\coordinate (cc) at (.75,.5);
\draw (a)--(0,0);
\draw (b)--(0,0);
\draw (c)--(0,0);
\draw (d)--(0,0);
\draw[fill=white] (a) circle (.25);
\draw[fill=white] (b) circle (.25);
\draw[fill=white] (c) circle (.25);
\draw[fill=white] (d) circle (.25);
\node at (0,1) { };
\node at (a) {$0$};
\node at (b) {$0$};
\node at (c) {$0$};
\node at (d) {$0$};
\path[fill=white] (0,0) circle (.2);
\node at (0,0) {$x$};
\end{tikzpicture} &
%%%%%%%%%%%%%%%%%%%%%%
\begin{tikzpicture}
\coordinate (a) at (-.75, 0);
\coordinate (b) at (0, .75);
\coordinate (c) at (.75, 0);
\coordinate (d) at (0, -.75);
\coordinate (aa) at (-.75,.5);
\coordinate (cc) at (.75,.5);
\draw[line width=0.5mm] (a)--(0,0);
\draw[line width=0.5mm] (b)--(0,0);
\draw[line width=0.5mm] (c)--(0,0);
\draw[line width=0.5mm] (d)--(0,0);
\draw[fill=white] (a) circle (.25);
\draw[fill=white] (b) circle (.25);
\draw[fill=white] (c) circle (.25);
\draw[fill=white] (d) circle (.25);
\node at (0,1) { };
\node at (a) {$1$};
\node at (b) {$1$};
\node at (c) {$1$};
\node at (d) {$1$};
\path[fill=white] (0,0) circle (.2);
\node at (0,0) {$x$};
\end{tikzpicture}&
%%%%%%%%%%%%%%%%%%%%%%
\begin{tikzpicture}
\coordinate (a) at (-.75, 0);
\coordinate (b) at (0, .75);
\coordinate (c) at (.75, 0);
\coordinate (d) at (0, -.75);
\coordinate (aa) at (-.75,.5);
\coordinate (cc) at (.75,.5);
\draw (a)--(0,0);
\draw[line width=0.5mm] (b)--(0,0);
\draw (c)--(0,0);
\draw[line width=0.5mm] (d)--(0,0);
\draw[fill=white] (a) circle (.25);
\draw[fill=white] (b) circle (.25);
\draw[fill=white] (c) circle (.25);
\draw[fill=white] (d) circle (.25);
\node at (0,1) { };
\node at (a) {$0$};
\node at (b) {$1$};
\node at (c) {$0$};
\node at (d) {$1$};
\path[fill=white] (0,0) circle (.2);
\node at (0,0) {$x$};
\end{tikzpicture}&
%%%%%%%%%%%%%%%%%%%%%%
\begin{tikzpicture}
\coordinate (a) at (-.75, 0);
\coordinate (b) at (0, .75);
\coordinate (c) at (.75, 0);
\coordinate (d) at (0, -.75);
\coordinate (aa) at (-.75,.5);
\coordinate (cc) at (.75,.5);
\draw[line width=0.5mm] (a)--(0,0);
\draw (b)--(0,0);
\draw[line width=0.5mm] (c)--(0,0);
\draw (d)--(0,0);
\draw[fill=white] (a) circle (.25);
\draw[fill=white] (b) circle (.25);
\draw[fill=white] (c) circle (.25);
\draw[fill=white] (d) circle (.25);
\node at (0,1) { };
\node at (a) {$1$};
\node at (b) {$0$};
\node at (c) {$1$};
\node at (d) {$0$};
\path[fill=white] (0,0) circle (.2);
\node at (0,0) {$x$};
\end{tikzpicture}&
%%%%%%%%%%%%%%%%%%%%%%
\begin{tikzpicture}
\coordinate (a) at (-.75, 0);
\coordinate (b) at (0, .75);
\coordinate (c) at (.75, 0);
\coordinate (d) at (0, -.75);
\coordinate (aa) at (-.75,.5);
\coordinate (cc) at (.75,.5);
\draw (a)--(0,0);
\draw[line width=0.5mm] (b)--(0,0);
\draw[line width=0.5mm] (c)--(0,0);
\draw (d)--(0,0);
\draw[fill=white] (a) circle (.25);
\draw[fill=white] (b) circle (.25);
\draw[fill=white] (c) circle (.25);
\draw[fill=white] (d) circle (.25);
\node at (0,1) { };
\node at (a) {$0$};
\node at (b) {$1$};
\node at (c) {$1$};
\node at (d) {$0$};
\path[fill=white] (0,0) circle (.2);
\node at (0,0) {$x$};
\end{tikzpicture}&
%%%%%%%%%%%%%%%%%%%%%%
\begin{tikzpicture}
\coordinate (a) at (-.75, 0);
\coordinate (b) at (0, .75);
\coordinate (c) at (.75, 0);
\coordinate (d) at (0, -.75);
\coordinate (aa) at (-.75,.5);
\coordinate (cc) at (.75,.5);
\draw[line width=0.5mm] (a)--(0,0);
\draw (b)--(0,0);
\draw (c)--(0,0);
\draw[line width=0.5mm] (d)--(0,0);
\draw[fill=white] (a) circle (.25);
\draw[fill=white] (b) circle (.25);
\draw[fill=white] (c) circle (.25);
\draw[fill=white] (d) circle (.25);
\node at (0,1) { };
\node at (a) {$1$};
\node at (b) {$0$};
\node at (c) {$0$};
\node at (d) {$1$};
\path[fill=white] (0,0) circle (.2);
\node at (0,0) {$x$};
\end{tikzpicture} \\
\midrule
  1 & 1 & 1 & \beta(x \oplus y) & 1 + \beta(x \oplus y) & 1 \\
\bottomrule
\end{array}\]
\caption{Boltzmann weights of the semidual model.}
\label{fig:dual_model_wts}
\end{figure}

We define the \defn{semidual model} $\Dstates$ to be the colored lattice model on an $n \times n$ grid using the $L$-matrix given in Figure~\ref{fig:dual_model_wts} and boundary conditions as follows.
The left (resp.\ bottom) boundary condition is the colors $c_n > \cdots > c_1$ from top-to-bottom (resp.\ left-to-right); the top and right boundary edges are all $0$.
Since (again) no colors cross, we can forget the coloring, but unlike before, we obtain an equivalent model as there is no dependency on $w$.
The following proposition is straightforward.

\begin{proposition}
\label{prop:map_semidual}
The map $\Phi \colon \Gstates \to \Dstates$ given by
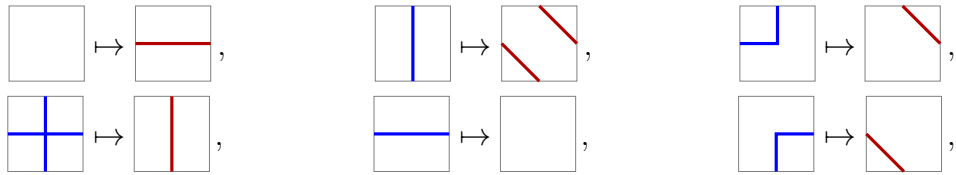
\begin{align*}
\begin{tikzpicture}[scale=0.5,baseline=10]
\draw[very thin,gray] (0,0) rectangle (2,2);
\end{tikzpicture}
& \mapsto
\begin{tikzpicture}[scale=0.5,baseline=10]
\draw[very thin,gray] (0,0) rectangle (2,2);
\draw[very thick, darkred] (0,1) -- (2,1);
\end{tikzpicture}\,,
&
\begin{tikzpicture}[scale=0.5,baseline=10]
\draw[very thin,gray] (0,0) rectangle (2,2);
\draw[very thick, blue] (1,0) -- (1,2);
\end{tikzpicture}
& \mapsto
\begin{tikzpicture}[scale=0.5,baseline=10]
\draw[very thin,gray] (0,0) rectangle (2,2);
\draw[very thick, darkred] (1,2) -- (2,1);
\draw[very thick, darkred] (0,1) -- (1,0);
\end{tikzpicture}\,,
&
\begin{tikzpicture}[scale=0.5,baseline=10]
\draw[very thin,gray] (0,0) rectangle (2,2);
\draw[very thick, blue] (0,1) -- (1,1) -- (1,2);
\end{tikzpicture}
& \mapsto
\begin{tikzpicture}[scale=0.5,baseline=10]
\draw[very thin,gray] (0,0) rectangle (2,2);
\draw[very thick, darkred] (1,2) -- (2,1);
\end{tikzpicture}\,,
\\
\begin{tikzpicture}[scale=0.5,baseline=10]
\draw[very thin,gray] (0,0) rectangle (2,2);
\draw[very thick, blue] (1,0) -- (1,2);
\draw[very thick, blue] (0,1) -- (2,1);
\end{tikzpicture}
& \mapsto
\begin{tikzpicture}[scale=0.5,baseline=10]
\draw[very thin,gray] (0,0) rectangle (2,2);
\draw[very thick, darkred] (1,0) -- (1,2);
\end{tikzpicture}\,,
&
\begin{tikzpicture}[scale=0.5,baseline=10]
\draw[very thin,gray] (0,0) rectangle (2,2);
\draw[very thick, blue] (0,1) -- (2,1);
\end{tikzpicture}
& \mapsto
\begin{tikzpicture}[scale=0.5,baseline=10]
\draw[very thin,gray] (0,0) rectangle (2,2);
\end{tikzpicture}\,,
&
\begin{tikzpicture}[scale=0.5,baseline=10]
\draw[very thin,gray] (0,0) rectangle (2,2);
\draw[very thick, blue] (1,0) -- (1,1) -- (2,1);
\end{tikzpicture}
& \mapsto
\begin{tikzpicture}[scale=0.5,baseline=10]
\draw[very thin,gray] (0,0) rectangle (2,2);
\draw[very thick, darkred] (0,1) -- (1,0);
\end{tikzpicture}\,,
\end{align*}
is a bijection.
\end{proposition}

We note that the map in Proposition~\ref{prop:map_semidual} is defined by interchanging $0 \leftrightarrow 1$ on the horizontal (auxiliary space) component between the models.
This is why we refer to this model as the semidual model.

\begin{example}
\label{ex:4x4_semidual}
We apply the bijection defined in Proposition~\ref{prop:map_semidual} to the state given in Figure~\ref{fig:colored_state}(right):
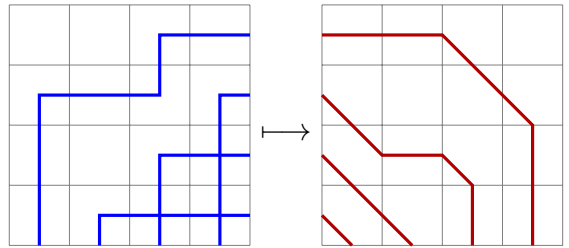
\[
\begin{tikzpicture}[scale=0.4,baseline=40]
%\fill[black!10] (2,0) -- (2,2) -- (4,2) -- (4,4) -- (6,4) -- (6,6) -- (8,6) -- (8,0) -- cycle;
\draw[very thin,gray] (0,0) grid[step=2] (8,8);
\draw[very thick, blue] (1,0) -- (1,5) -- (5,5) -- (5,7) -- (8,7);
\draw[very thick, blue] (3,0) -- (3,1) -- (8,1);
\draw[very thick, blue] (5,0) -- (5,3) -- (8,3);
\draw[very thick, blue] (7,0) -- (7,5) -- (8,5);
\end{tikzpicture}
\longmapsto
\begin{tikzpicture}[scale=0.4,baseline=40]
%\fill[black!10] (2,0) -- (2,2) -- (4,2) -- (4,4) -- (6,4) -- (6,6) -- (8,6) -- (8,0) -- cycle;
\draw[very thin,gray] (0,0) grid[step=2] (8,8);
\draw[very thick, darkred] (0,7) -- (4,7) -- (7,4) -- (7,0);
\draw[very thick, darkred] (0,5) -- (2,3) -- (4,3) -- (5,2) -- (5,0);
\draw[very thick, darkred] (0,3) -- (3,0);
\draw[very thick, darkred] (0,1) -- (1,0);
\end{tikzpicture}\,.
\]
\end{example}

\begin{figure}
\[
\begin{array}{c@{\qquad}c@{\qquad}c@{\qquad}c@{\qquad}c}
\toprule
\begin{tikzpicture}[scale=0.7]
\draw (0,0) to [out = 0, in = 180] (2,2);
\draw (0,2) to [out = 0, in = 180] (2,0);
\draw[fill=white] (0,0) circle (.35);
\draw[fill=white] (0,2) circle (.35);
\draw[fill=white] (2,0) circle (.35);
\draw[fill=white] (2,2) circle (.35);
\node at (0,0) {$1$};
\node at (0,2) {$1$};
\node at (2,2) {$1$};
\node at (2,0) {$1$};
\node at (2,0) {$1$};
\path[fill=white] (1,1) circle (.3);
\node at (1,1) {$x_i,x_j$};
\end{tikzpicture}&
\begin{tikzpicture}[scale=0.7]
\draw (0,0) to [out = 0, in = 180] (2,2);
\draw (0,2) to [out = 0, in = 180] (2,0);
\draw[fill=white] (0,0) circle (.35);
\draw[fill=white] (0,2) circle (.35);
\draw[fill=white] (2,0) circle (.35);
\draw[fill=white] (2,2) circle (.35);
\node at (0,0) {$0$};
\node at (0,2) {$0$};
\node at (2,2) {$0$};
\node at (2,0) {$0$};
\path[fill=white] (1,1) circle (.3);
\node at (1,1) {$x_i,x_j$};
\end{tikzpicture}&
\begin{tikzpicture}[scale=0.7]
\draw (0,0) to [out = 0, in = 180] (2,2);
\draw (0,2) to [out = 0, in = 180] (2,0);
\draw[fill=white] (0,0) circle (.35);
\draw[fill=white] (0,2) circle (.35);
\draw[fill=white] (2,0) circle (.35);
\draw[fill=white] (2,2) circle (.35);
\node at (0,0) {$1$};
\node at (0,2) {$0$};
\node at (2,2) {$1$};
\node at (2,0) {$0$};
\path[fill=white] (1,1) circle (.3);
\node at (1,1) {$x_i,x_j$};
\end{tikzpicture}&
\begin{tikzpicture}[scale=0.7]
\draw (0,0) to [out = 0, in = 180] (2,2);
\draw (0,2) to [out = 0, in = 180] (2,0);
\draw[fill=white] (0,0) circle (.35);
\draw[fill=white] (0,2) circle (.35);
\draw[fill=white] (2,0) circle (.35);
\draw[fill=white] (2,2) circle (.35);
\node at (0,0) {$0$};
\node at (0,2) {$1$};
\node at (2,2) {$1$};
\node at (2,0) {$0$};
\path[fill=white] (1,1) circle (.3);
\node at (1,1) {$x_i,x_j$};
\end{tikzpicture}&
\begin{tikzpicture}[scale=0.7]
\draw (0,0) to [out = 0, in = 180] (2,2);
\draw (0,2) to [out = 0, in = 180] (2,0);
\draw[fill=white] (0,0) circle (.35);
\draw[fill=white] (0,2) circle (.35);
\draw[fill=white] (2,0) circle (.35);
\draw[fill=white] (2,2) circle (.35);
\node at (0,0) {$1$};
\node at (0,2) {$0$};
\node at (2,2) {$0$};
\node at (2,0) {$1$};
\path[fill=white] (1,1) circle (.3);
\node at (1,1) {$x_i,x_j$};
\end{tikzpicture}\\
\midrule
1 + \beta x_j & 1 + \beta x_i & \beta (x_j - x_i) & 1 + \beta x_j & 1 + \beta x_i \\
\bottomrule
\end{array}
\]
\caption{The $R$-matrix for the semidual model.}
\label{fig:dual_R_matrix}
\end{figure}
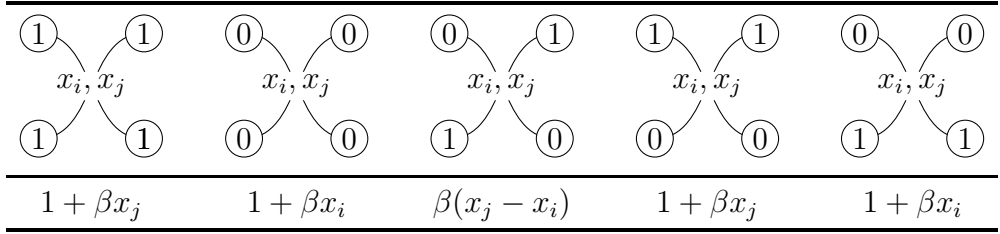

\begin{proposition}
The uncolored semidual model is an integrable model with the $R$-matrix not depending on the parameters $\yy$.
\end{proposition}

\begin{proof}
This is a finite computation using the $R$-matrix given by Figure~\ref{fig:dual_R_matrix}.
\end{proof}

Now we want to reintroduce the vexillary permutation $w$ into these uncolored models.
If we restrict our model to $\Lambda_w$ (as opposed to an $n \times n$ grid), we can use~\cite[Lemma~7.2]{Weigandt20}\footnote{
   This can also be seen as a translation of Proposition~\ref{prop:vexillary_classification} into bumpless pipe dreams.
}
to see that there are no vertices of type $\tt{a}_2$ nor $\tt{a}_2^{\dagger}$ in $\Gstates_w$.
Therefore we have an uncolored five-vertex model whose partition function is $\beta^{\ell(w)} \G_w(\xx, \yy; \beta)$.
Equivalently for the corresponding semidual model on $\Lambda_w$ there are no vertical lines, and the states can be thought of as a family of nonintersecting lattice paths.
This allows us to change the $L$-matrix to the (uncolored) one from Figure~\ref{fig:colored_weights_Lpoly} but rotated by 180 degrees (these are also the tiles in~\cite[Eq.~(7)]{WZJ19} rotated 180 degrees and reflected across the vertical axis with $z \mapsto 1 - (x \oplus y)$ at $\beta = -1$).
In terms of the partition function, all this does is remove the $\beta^{\ell(w)}$ as only the weight of vertex $\overline{\tt{b}}_2$ changes by removing the $\beta$, of which there are precisely $\ell(w)$ such vertices.
Hence we have an integrable model, which we denote by $\Dstates_w$, and the following result.

\begin{proposition}
The semidual model $\Dstates_w$ is an integrable model whose $R$-matrix also does not depend on $\yy$.
Furthermore, the partition function is given by
\[
\beta^{\ell(w)} Z(\Dstates_w; \xx, \yy; \beta) = Z(\Gstates_w; \xx, \yy; \beta).
\]
\end{proposition}

Using the semidual model $\Dstates_w$, let $F_w := (F_i)_{i=1}^k$ denote the heights of the right endpoints of the paths in $\Lambda_w$ that end on the right boundary.
So the tile immedately to the right of such a boundary point is a
$
\begin{tikzpicture}[scale=0.25,baseline=4]
\draw[very thin,gray,fill=black!10] (0,0) rectangle (2,2);
\draw[thick, darkred] (0,1) -- (1,0);
\end{tikzpicture}
$
not in $\Lambda_w$.
Note that any other path that does not has such a right endpoint must simply move diagonally.
Furthermore, the sequence $F_w$ is exactly the flagging associated to $w$.

\begin{example}
We apply the bijection $\Phi$ from Proposition~\ref{prop:map_semidual} to the vexillary permutation $w = [8,7,1,6,2,9,5,3,4]$ (written in one-line notation) % $w^{-1} = [3,5,8,9,7,4,2,1,6]$
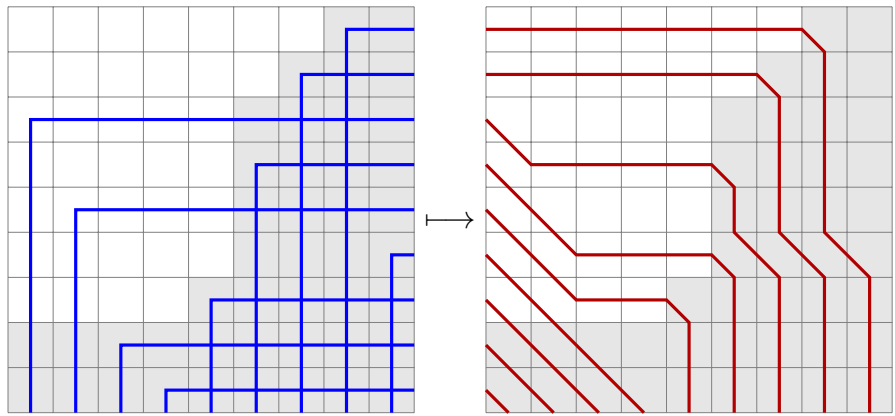
\[
\begin{tikzpicture}[scale=0.3,baseline=70]
\fill[black!10] (0,0) -- (0,4) -- (8,4) -- (8,6) -- (10,6) -- (10,14) -- (12,14) -- (12,16) -- (14,16) -- (14,18) -- (18,18) -- (18,0) -- cycle;
\draw[very thin,gray] (0,0) grid[step=2] (18,18);
\foreach \y/\x in {1/8,2/7,3/1,4/6,5/2,6/9,7/5,8/3,9/4} {
    \draw[very thick, blue] (2*\x-1,0) -- (2*\x-1,18-2*\y+1) -- (18,18-2*\y+1);
}
\end{tikzpicture}
\longmapsto
\begin{tikzpicture}[scale=0.3,baseline=70]
\fill[black!10] (0,0) -- (0,4) -- (8,4) -- (8,6) -- (10,6) -- (10,14) -- (12,14) -- (12,16) -- (14,16) -- (14,18) -- (18,18) -- (18,0) -- cycle;
\draw[very thin,gray] (0,0) grid[step=2] (18,18);
\draw[very thick, darkred] (0,17) -- (14,17) -- (15,16) -- (15,8) -- (17,6) -- (17,0);
\draw[very thick, darkred] (0,15) -- (12,15) -- (13,14) -- (13,8) -- (15,6) -- (15,0);
\draw[very thick, darkred] (0,13) -- (2,11) -- (10,11) -- (11,10) -- (11,8) -- (13,6) -- (13,0);
\draw[very thick, darkred] (0,11) -- (4,7) -- (10,7) -- (11,6) -- (11,0);
\draw[very thick, darkred] (0,9) -- (4,5) -- (8,5) -- (9,4) -- (9,0);
\foreach \x in {7,5,3,1}
    \draw[very thick, darkred] (0,\x) -- (\x,0);
\end{tikzpicture}\,,
\]
where we have shaded the tiles that do not belong the partition $\Lambda_w$.
The heights of the boundary points are $F_w = (1, 2, 4, 6, 7)$  (cf.~\cite[Ex.~5.7]{KMY09}).
\end{example}

Now if we rotate the semidual model $\Dstates_w$ by 180 degrees and extend from the endpoints in the (rotated) shape $\Lambda_w$ by diagonal lines (which go to the northwest), we obtain precisely a state in the uncolored model $\states_{\lambda_w}$ (see also~\cite{GK17,MS13,Motegi20}).
Thus we can apply the natural bijection $\Theta$ between marked states, where we allow only tiles
$
\begin{tikzpicture}[scale=0.25,baseline=4]
\draw (0,0) rectangle (2,2);
\draw[thick, darkred] (1,2) -- (2,1);
\end{tikzpicture}
$
to be marked, and set-valued tableaux via marked Gelfand--Tsetlin patterns from~\cite[Sec.~2.3]{BSW20}.
However, the flagging $F_w$ imposes a restriction on the possible states and set-valued tableaux.
It is straightforward to see that this precisely corresponds to restricting to the set of flagged set-valued tableau $\svt_{\lambda_w,F_w}$.
Hence, we have a new proof of Theorem~\ref{thm:vexillary_flagged}.

\begin{theorem}[{\cite[Thm.~5.8]{KMY09}}]
\label{thm:vexillary_flagged}
Let $w \in \sym{n}$ be a vexillary permutation. Then we have
\[
\G_w(\xx, \yy; \beta) = Z(\Dstates_w; \xx, \yy; \beta) = \G_{\lambda_w,F_w}(\xx | \yy; \beta).
\]
\end{theorem}

As in~\cite[Cor.~5.9]{KMY09}, this gives an integrable systems proof of~\cite[Thm~3.1]{Buch02}.
The bijection $\Theta$ is also the same as the bijection from pipe dreams to $\svt_{\lambda_w,F_w}$ given in~\cite[Prop.~5.3]{KMY09}, which is different from the formula in~\cite{FK94,FK96} (see, \textit{e.g.},~\cite[Ex.~5.10]{KMY09}).

We can also give new a proof of~\cite[Thm.~8.7]{McNamara06} and Theorem~\ref{thm:uncolored_factorial} by taking the stable limit of the semidual model $\Dstates_{1^k \times w}$ as $k \to \infty$ and restricting to a finite number of variables.

\begin{theorem}[{\cite{GK17,MS13,Motegi20,WZJ19}}]
\label{thm:factorial_model}
Fix a positive integer $n$.
Let $w \in \sym{m}$ be a vexillary permutation with $\lambda = \lambda_w$ such that $m > n + \lambda_1$.
Then we have
\[
Z(\Dstates_w; \xx_n, \yy; \beta) = Z(\states_{\lambda_w}; \xx, \yy; \beta) = \G_{\lambda_w}(\xx_n|\yy; \beta) = \lim_{k\to\infty} \G_{1^k \times w}(\xx_{m+k}, \yy; \beta) \Bigr\rvert_{\xx_n},
\]
where $\xx_n = (x_1, \dotsc, x_n, 0, 0, \ldots)$.
% is a factorial Grothendieck polynomial in $n$ variables.
\end{theorem}

\begin{proof}
It is well known that $\lambda_{1^k \times w} = \lambda_w$, and for $F_{1^k \times w} = (F'_i)_{i=1}^{\ell(\lambda_w)}$ and $F_w = (F_i)_{i=1}^{\ell(\lambda_w)}$, we have $F_i' = F_i + k$.
This is also easy to see from examining the corresponding vertex models (equivalently Rothe diagrams).
By taking $k$ such that $F_1' > n$, the flagging condition no longer applies as we cannot have an entry in the tableau larger than $n$ and the flagging is strictly increasing.
Thus the model is exactly the uncolored five-vertex model $\states_{\lambda_w}$, and the claim follows.
\end{proof}

\begin{example}
We consider $\lambda = (1)$ and $n = 2$. We consider $w = [2, 1]$ and $1^2 \times w$:
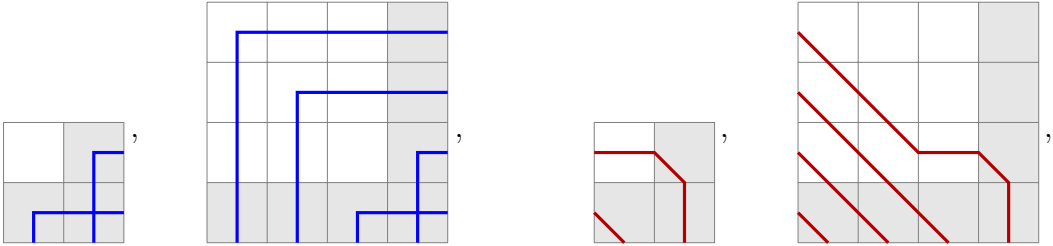
\[
\begin{tikzpicture}[scale=0.4,baseline=40]
\fill[black!10] (0,0) -- (0,2) -- (2,2) -- (2,4) -- (4,4) -- (4,0) -- cycle;
\draw[very thin,gray] (0,0) grid[step=2] (4,4);
\draw[very thick, blue] (1,0) -- (1,1) -- (4,1);
\draw[very thick, blue] (3,0) -- (3,3) -- (4,3);
\end{tikzpicture}\,,
\qquad
\begin{tikzpicture}[scale=0.4,baseline=40]
\fill[black!10] (0,0) -- (0,2) -- (6,2) -- (6,8) -- (8,8) -- (8,0) -- cycle;
\draw[very thin,gray] (0,0) grid[step=2] (8,8);
\draw[very thick, blue] (5,0) -- (5,1) -- (8,1);
\draw[very thick, blue] (7,0) -- (7,3) -- (8,3);
\draw[very thick, blue] (3,0) -- (3,5) -- (8,5);
\draw[very thick, blue] (1,0) -- (1,7) -- (8,7);
\end{tikzpicture}\,,
\qquad\qquad
\begin{tikzpicture}[scale=0.4,baseline=40]
\fill[black!10] (0,0) -- (0,2) -- (2,2) -- (2,4) -- (4,4) -- (4,0) -- cycle;
\draw[very thin,gray] (0,0) grid[step=2] (4,4);
\draw[very thick, darkred] (0,1) -- (1,0);
\draw[very thick, darkred] (0,3) -- (2,3) -- (3,2) -- (3,0);
\end{tikzpicture}\,,
\qquad
\begin{tikzpicture}[scale=0.4,baseline=40]
\fill[black!10] (0,0) -- (0,2) -- (6,2) -- (6,8) -- (8,8) -- (8,0) -- cycle;
\draw[very thin,gray] (0,0) grid[step=2] (8,8);
\draw[very thick, darkred] (0,1) -- (1,0);
\draw[very thick, darkred] (0,3) -- (3,0);
\draw[very thick, darkred] (0,5) -- (5,0);
\draw[very thick, darkred] (0,7) -- (4,3) -- (6,3) -- (7,2) -- (7,0);
\end{tikzpicture}\,,
\]
where we have drawn the normal double Grothendieck model states on the left and the corresponding semidual states on the right.
If we restrict to the upper-left $2 \times 3$ rectangle, we see that every state of $\states_{\lambda}$ (after being rotated by 180 degrees) can be realized inside of that $2 \times 3$ rectangle:
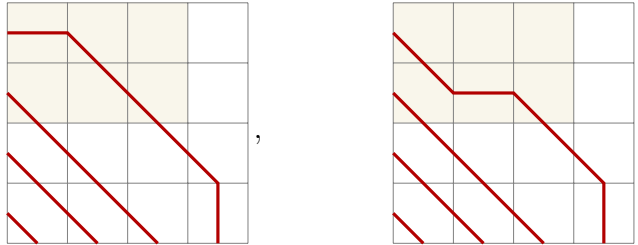
\[
\begin{tikzpicture}[scale=0.4,baseline=40]
\fill[UQgold!10] (0,8) rectangle (6,4);
\draw[very thin,gray] (0,0) grid[step=2] (8,8);
\draw[very thick, darkred] (0,1) -- (1,0);
\draw[very thick, darkred] (0,3) -- (3,0);
\draw[very thick, darkred] (0,5) -- (5,0);
\draw[very thick, darkred] (0,7) -- (2,7) -- (7,2) -- (7,0);
\end{tikzpicture}\,,
\qquad\qquad
\begin{tikzpicture}[scale=0.4,baseline=40]
\fill[UQgold!10] (0,8) rectangle (6,4);
\draw[very thin,gray] (0,0) grid[step=2] (8,8);
\draw[very thick, darkred] (0,1) -- (1,0);
\draw[very thick, darkred] (0,3) -- (3,0);
\draw[very thick, darkred] (0,5) -- (5,0);
\draw[very thick, darkred] (0,7) -- (2,5) -- (4,5) -- (7,2) -- (7,0);
\end{tikzpicture}\,.
\]
\end{example}

We also note that the symmetry used in~\cite[Lemma~8.1]{Weigandt20} is precisely the natural symmetry from reflecting the semidual model along the $y = x$ line.
Furthermore, the bijection in \cite[Lemma~8.2]{Weigandt20} comes from the reflected elementary excitations/emissions, where we instead consider the paths in $\Lambda_w$ (which in this case equals $\lambda_w$) as being fixed.

%%%%%%%%%%%%%%%%%%%%
\subsection{Excited Young diagrams}

We can relate a state of our models to \defn{excited Young diagram} (EYD), a combinatorial object introduced independently in~\cite{Kreiman05} and~\cite{IN09} to do computations in equivariant Schubert calculus.
We use the extension for K-theory from~\cite{IN13,GK15}.
Let $\lambda \subseteq \mu$.
The set of excited Young diagrams $\eyd_{\lambda,\mu}$ is a subset of boxes of the Young diagram of $\mu$ generated from $\lambda$ by the local moves
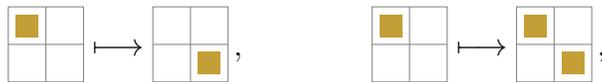
\[
\begin{tikzpicture}[scale=0.5,baseline=10]
\draw[very thin,gray] (0,0) rectangle (2,2);
\draw[very thin,gray] (0,1) -- (2,1);
\draw[very thin,gray] (1,0) -- (1,2);
\fill[UQgold] (0.2,1+0.2) rectangle (0.8,1+0.8);
\end{tikzpicture}
\longmapsto
\begin{tikzpicture}[scale=0.5,baseline=10]
\draw[very thin,gray] (0,0) rectangle (2,2);
\draw[very thin,gray] (0,1) -- (2,1);
\draw[very thin,gray] (1,0) -- (1,2);
\fill[UQgold] (1+0.2,0.2) rectangle (1+0.8,0.8);
\end{tikzpicture}\,,
\qquad\qquad
\begin{tikzpicture}[scale=0.5,baseline=10]
\draw[very thin,gray] (0,0) rectangle (2,2);
\draw[very thin,gray] (0,1) -- (2,1);
\draw[very thin,gray] (1,0) -- (1,2);
\fill[UQgold] (0.2,1+0.2) rectangle (0.8,1+0.8);
\end{tikzpicture}
\longmapsto
\begin{tikzpicture}[scale=0.5,baseline=10]
\draw[very thin,gray] (0,0) rectangle (2,2);
\draw[very thin,gray] (0,1) -- (2,1);
\draw[very thin,gray] (1,0) -- (1,2);
\fill[UQgold] (0.2,1+0.2) rectangle (0.8,1+0.8);
\fill[UQgold] (1+0.2,0.2) rectangle (1+0.8,0.8);
\end{tikzpicture}\,,
\]
called \defn{elementary excitations}.
Following~\cite{MPS18}, we call the inverses \defn{elementary emissions}.

In terms of the double Grothendieck model, an excited Young diagram consists of the tiles for $\tt{a}_1$ vertices and marked $\tt{c}_1$ vertices.
It is straightforward to see this is an equivalent definition: The local moves and their inverses from~\cite[Eq.~(7.1), Eq.~(7.2)]{Weigandt20} (with consideration of the marked corners) are precisely the elementary excitations and emissions.
Using the semidual model, an excited Young diagram can be defined as the set of
$
\begin{tikzpicture}[scale=0.25,baseline=4]
\draw[very thin,gray] (0,0) rectangle (2,2);
\draw[thick, darkred] (0,1) -- (2,1);
\end{tikzpicture}
$
tiles and marked
$
\begin{tikzpicture}[scale=0.25,baseline=4]
\draw[very thin,gray] (0,0) rectangle (2,2);
\draw[thick, darkred] (1,2) -- (2,1);
\end{tikzpicture}
$
tiles.
It is also straightforward to see that elementary excitations correspond to moving the
$
\begin{tikzpicture}[scale=0.25,baseline=4]
\draw[very thin,gray] (0,0) rectangle (2,2);
\draw[thick, darkred] (0,1) -- (2,1);
\end{tikzpicture}
$
down diagonally one step (provided there is not a marked tile there), adding a particular marking, or moving markings along a diagonal (when possible).
Summarizing this, we have the following proposition.

\begin{proposition}
The maps
\[
\Gamma_{\Gstates_w} \colon \Gstates_w \to \eyd_{\lambda_w,\Lambda_w},
\qquad\qquad
\Gamma_{\Dstates_w} \colon \Dstates_w \to \eyd_{\lambda_w,\Lambda_w},
\]
described above are weight preserving bijections.
\end{proposition}

Let $\Psi$ denote the bijection between excited Young diagrams and set-valued tableaux from~\cite{GK15}.
Recall that $\Theta \colon \Dstates_w \to \svt_{\lambda_w,F_w}$ is the restriction of the natural bijection from~\cite[Sec.~2.3]{BSW20}.

\begin{proposition}
Let $w$ be a vexillary permutation.
Then the diagram
\[
\xymatrix@R=3em@C=4em{
\Gstates_w \ar[r]^{\Phi} \ar[d]_{\Gamma_{\Gstates_w}} & \Dstates_w \ar[dl]_{\Gamma_{\Dstates_w}} \ar[d]^{\Theta} \\
\eyd_{\lambda_w,\Lambda_w} \ar[r]_{\Psi} & \svt_{\lambda_w,F_w}
}
\]
commutes.
\end{proposition}

\begin{proof}
This can be seen by considering how elementary excitations commute under each of these bijections.
\end{proof}

%%%%%%%%%%%%%%%%%%%%
\subsection{Nonintersecting lattice paths}

There is another benefit with the semidual model and the interpretation of a state as a family of nonintersecting lattice paths.
The \defn{Lindstr\"om--Gessel--Viennot (LGV) lemma}~\cite{Lindstrom73,GV85} posits that the sum of the weights over all families of nonintersecting lattice paths in a edge-weighted directed graph can be given as a determinant of the matrix $\begin{bmatrix} p_{ab} \end{bmatrix}_{a,b=1}^n$, where $p_{a,b}$ is the sum of the weights of each path (which is the product of the edge weights in that path) from the $a$-th starting point to the $b$-th ending point.
We will construct a weight preserving bijection from a marked state of the semidual model into a family of nonintersecting lattice paths in order to use the LGV lemma to express the partition function as a determinant.

We use the following local translation from tiles to a directed graph:
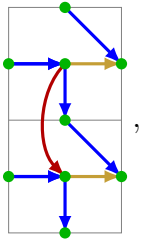
\[
\begin{tikzpicture}[scale=0.75,baseline=40,>=latex]
\draw[very thin,gray] (0,0) rectangle (2,4);
\draw[very thin,gray] (0,2) -- (2,2);
\draw[very thick, darkred, ->] (1,3) .. controls (0.5,2.5) and (0.5,1.5) .. (1,1);
\foreach \y in {0,2} {
\draw[very thick, blue, ->] (0,\y+1) -- (1,\y+1);
\draw[very thick, blue, ->] (0,\y+1) -- (1,\y+1);
\draw[very thick, UQgold, ->] (1,\y+1) -- (2,\y+1);
\draw[very thick, blue, ->] (1,\y+1) -- (1,\y+0);
\draw[very thick, blue, ->] (1,\y+2) -- (2,\y+1);
\fill[darkgreen] (0,\y+1) circle (0.1);
\fill[darkgreen] (1,\y+1) circle (0.1);
\fill[darkgreen] (2,\y+1) circle (0.1);
\fill[darkgreen] (1,\y+0) circle (0.1);
}
\fill[darkgreen] (1,4) circle (0.1);
\end{tikzpicture}\,,
\]
with a non-trivial edge weight on only the second horizontal step in each tile that depends on the position of the tile.
Formally, we have a graph on the vertex set
\[
V = \bigl\{ (2i, 2j+1) \mid i, j \in \{1, \dotsc, n\} \bigr\} \cup \bigl\{ (2i+1,j) \mid i \in \{1,\dotsc,n\}, j \in \{1, \dotsc, 2n+1\} \bigr\},
\]
with the edge set divided into two types of edges: \defn{Schubert edges}
\[
(i, 2j+1) \to (i+1, 2j+1),
\quad
(2i+1, 2j) \to (2i+2, 2j+1),
\quad
(2i+1, 2j+1) \to (2i+1, 2j),
\]
and \defn{K-theory edges} $(2i+1, 2j+1) \to (2i+1, 2j-1)$ for all appropriate $i$ and $j$.
We draw this graph following English convention (so $(x,y) \mapsto (x, -y)$ in our drawing convention) to match with the tile description above.
All edges have weight $1$ except for the Schubert edges $(2i, 2j+1) \to (2i+1, 2j+1)$, which have weight $w(i, j)$ that we will give later.
We can also restrict to the paths that are not simply diagonals (\textit{i.e.}, paths that have at least one horizontal step).

We first consider what happens when we apply the LGV lemma na\"ively using the tiles.
In this case, we see that we must have $\beta = 0$, and so we take $w(i, j) = x_i \oplus y_j$, remove the K-theory edges, and only consider unmarked states.
Let $h_b$ denote the height of the $b$-th endpoint and $\lambda = \lambda_w$.
One can see that
\begin{equation}
\label{eq:path_weight_Schubert}
p_{ab} = \G_{(\lambda_b + a - b)}(x_a, \dotsc, x_{h_b} | \yy; 0).
% partition part is (\lambda_b + (h_b - b) - (h_b - a)) since (h_b - b) is the number of extra horizontal steps to add to \lambda_b and (h_b - a) is those to subtract
\end{equation}
Therefore, we have the following expression for double Schubert polynomials (after substituting $\yy \mapsto -\yy$) by the LGV lemma.

\begin{theorem}
\label{thm:LGV_Schubert}
Let $w$ be a vexillary permutation.
Then we have
\[
Z(\Dstates_w; \xx, \yy; 0) = \det \begin{bmatrix} p_{ab} \end{bmatrix}_{a,b=1}^n = \G_w(\xx, \yy; 0),
\]
where $p_{ab}$ is given by Equation~\eqref{eq:path_weight_Schubert}.
\end{theorem}

We note that our formula at $\yy = 0$ is not the same as in~\cite{Wachs85}, but it is likely equivalent by some sequence of row operations.

\begin{example}
Consider the permutation $w = s_2 s_3 s_2$ (for an example of a state in $\Dstates_w$, see Example~\ref{ex:4x4_semidual}).
We compute $\lambda_w = (2,1)$ and $F_w = (2, 3)$.
Applying Theorem~\ref{thm:LGV_Schubert} in this case, we have
\begin{align*}
%b_1 = 2,   h_2 = 3,
Z(\Dstates_w; \xx, \yy; 0) & = \det \begin{bmatrix}
\G_{(2)}(x_1, x_2 | \yy; 0) & 1 & 0 & 0 \\
\G_{(3)}(x_2 | \yy; 0) & \G_{(1)}( x_2, x_3 | \yy; 0) & 0 & 0 \\
0 & \G_{(2)}( x_3 | \yy; 0) & 1 & 0 \\
0 & 0 & \G_{(1)}( x_4 | \yy; 0) & 1
\end{bmatrix}
%\\ & = \det \begin{bmatrix}
%\G_{(2)}(x_1, x_2 | \yy; 0) & 1 & 0 & 0 \\
%\prod_{j=1}^3 x_2 \oplus y_j & x_2 \oplus y_1 + x_3 \oplus y_2& 0 & 0 \\
%0 & (x_3 \oplus y_1) (x_3 \oplus y_2) & 1 & 0 \\
%0 & 0 & x_4 \oplus y_1 & 1
%\end{bmatrix}
\\ & = \det \begin{bmatrix}
\G_{(2)}(x_1, x_2 | \yy; 0) & 1 \\
\prod_{j=1}^3 x_2 \oplus y_j & x_2 \oplus y_1 + x_3 \oplus y_2 \\
\end{bmatrix},
\end{align*}
where
\[
\G_{(2)}(x_1, x_2 | \yy; 0) = (x_1 \oplus y_1)(x_1 \oplus y_2) + (x_1 \oplus y_1)(x_2 \oplus y_3) + (x_2 \oplus y_2)(x_2 \oplus y_3).
\]
Note that going from the $4 \times 4$ determinant to the $2 \times 2$ determinant comes from the fact that two of the paths contain no horizontal steps.
Hence, we have
\begin{align*}
Z(\Dstates_w; \xx, \yy; 0) & =
(x_1 \oplus y_1)(x_1 \oplus y_2)(x_2 \oplus y_1) + (x_1 \oplus y_1)(x_2 \oplus y_3)(x_2 \oplus y_1)
\\ & \hspace{20pt} + (x_1 \oplus y_1)(x_1 \oplus y_2)(x_3 \oplus y_2) + (x_1 \oplus y_1)(x_2 \oplus y_3)(x_3 \oplus y_2)
\\ & \hspace{20pt} + (x_2 \oplus y_2)(x_2 \oplus y_3)(x_3 \oplus y_2)
\\ & = \G_{\lambda_w, F_w}(x_1, x_2 | \yy; 0),
\end{align*}
with the corresponding flagged set-valued tableaux\footnote{The $\beta = 0$ condition means these are semistandard Young tableaux, which means no entry has more than one element in the corresponding set.}
\[
\ytableaushort{11,2}\,,
\qquad
\ytableaushort{12,2}\,,
\qquad
\ytableaushort{11,3}\,,
\qquad
\ytableaushort{12,3}\,,
\qquad
\ytableaushort{22,3}\,.
\]
If we compare this with the formula from~\cite[Thm.~1.3]{Wachs85}, we have
\begin{align*}
% p_{ab} = h_{\lambda_b - b + a}(F_b)
\G_{\lambda_w, F_w}(\xx | 0; 0) & =
\det \begin{bmatrix}
%h_{2}(x_1, x_2) & 1 \\
%h_{3}(x_1, x_2) & h_{1}(x_1,x_2,,x_3)
x_1^2 + x_1 x_2 + x_2^2 & 1 \\
x_1^3 + x_1^2 x_2 + x_1 x_2^2 + x_2^3 & x_1 + x_2 + x_3
\end{bmatrix},
%\\ & =  x_1^3 + x_1^2 x_2 + x_1 x_2^2 + x_1^2 x_2 + x_1 x_2^2 + x_2^3 + x_1^2 x_3 + x_1 x_2 x_3 + x_2^2 x_3
%\\ & =  x_1^3 + 2 x_1^2 x_2 + 2 x_1 x_2^2 + x_2^3 + x_1^2 x_3 + x_1 x_2 x_3 + x_2^2 x_3
%\\ & \hspace{20pt} - (x_1^3 + x_1^2 x_2 + x_1 x_2^2 + x_2^3)
\\ & =  x_1^2 x_2 + x_1 x_2^2 + x_1^2 x_3 + x_1 x_2 x_3 + x_2^2 x_3.
\end{align*}
Note that subtracting $x_1$ times the first row from the second in the above matrix is precisely the $2 \times 2$ determinant used to compute $Z(\Dstates_w; \xx, 0; 0)$.
\end{example}

Now we want to consider the case for generic $\beta$.
Here, we take $w(i,j) = \beta( x_i \oplus y_j)$, and we will see that we get a bijection with states using the $L$-matrix from Figure~\ref{fig:dual_model_wts}.
It is straightforward to see the only way you can travel along a K-theory edge is if you have a local configuration in the semidual model of
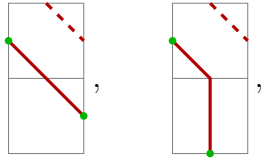
\[
\begin{tikzpicture}[scale=0.5,baseline=25]
\draw[very thin,gray] (0,0) rectangle (2,4);
\draw[very thin,gray] (0,2) -- (2,2);
\draw[very thick, darkred] (0,3) -- (2,1);
\draw[very thick, dashed, darkred] (1,4) -- (2,3);
\fill[darkgreen] (0,3) circle (0.1);
\fill[darkgreen] (2,1) circle (0.1);
\end{tikzpicture}\,,
\qquad
\begin{tikzpicture}[scale=0.5,baseline=25]
\draw[very thin,gray] (0,0) rectangle (2,4);
\draw[very thin,gray] (0,2) -- (2,2);
\draw[very thick, darkred] (0,3) -- (1,2) -- (1,0);
\draw[very thick, dashed, darkred] (1,4) -- (2,3);
\fill[darkgreen] (0,3) circle (0.1);
\fill[darkgreen] (1,0) circle (0.1);
\end{tikzpicture}\,,
\]
with the dashed edge possibly being in the top tile.
We note that there are two possible paths from the starting point to the end point for the left configuration, which corresponds to the weight $1 + \beta ( x_i \oplus y_j )$.
In particular, a tile is marked if and only if the lattice path uses the K-theory edge into that tile and then the path goes right.

However, this also means that we cannot consider $\Dstates_w$ (and hence obtain a result for $\G_w(\xx,\yy; \beta)$) but we can compute the partition function for $\Dstates$.
Thus, we can apply the LGV lemma, but we need to compute $p_{ab}$ that goes from the point $(0,2a-1)$ to $(2b-1,0)$.
We note the path is uniquely determined by the positions of the
$
\begin{tikzpicture}[scale=0.25,baseline=4]
\draw[very thin,gray] (0,0) rectangle (2,2);
\draw[thick, darkred] (1,2) -- (2,1);
\end{tikzpicture}
$
tiles and the weight can be explicitly computed.
We leave it as an exercise to the interested reader to give a precise formula.
Therefore, we have the following from the LGV lemma.

\begin{theorem}
We have
\[
Z(\Dstates; \xx, \yy; \beta) = \det \begin{bmatrix} p_{ab} \end{bmatrix}_{a,b=1}^n,
\]
where $p_{ab}$ is the sum of the weights over all paths from $(0,2a-1)$ to $(2b-1,0)$ in the graph given above.
\end{theorem}

\appendix
%===============================================================================
\section{\texorpdfstring{\textsc{SageMath}}{SageMath} code to obtain the \texorpdfstring{$R$}{R}-matrix}
\label{sec:sage_code}

We give the \textsc{SageMath}~\cite{sage} code we used to compute the $R$-matrix such that Proposition~\ref{prop:YBE} holds.

\lstset{numbers=left}
\begin{lstlisting}
def compute_R_matrix():
    R = ZZ['y,beta']
    y,beta = R.gens()
    def L_wt(aux_in, q_in, aux_out, q_out, z):
        if set([aux_in, q_in]) != set([aux_out, q_out]):
            return 0
        if aux_in == aux_out == q_out == q_in == 0:  # blank tile
            return beta*(z + y + beta*z*y)
        if aux_in == q_out == 0:  # right turn corner
            return 1
        if q_in == aux_out == 0:  # left turn corner
            return 1 + beta*(z + y + beta*z*y)
        if q_in == q_out == 0:  # horizontal line
            return 1
        if aux_in == aux_out == 0:  # vertical line
            return 1
        # Must be a crossing
        if aux_in == aux_out:
            if q_in < aux_in:
                return 1
        elif aux_in == q_out:
            if q_in > aux_in:
                return 1
        return 0

    states_to_vars = {(in_top,in_bot,out_top,out_bot): 0
                      for in_top in range(4)
                      for in_bot in range(4)
                      for out_top in range(4)
                      for out_bot in range(4)
                      if set([in_top, in_bot]) == set([out_top, out_bot])}
    zi, zj = R['zi,zj'].fraction_field().gens()
    base = zi.parent()
    S = PolynomialRing(base, 'x', len(states_to_vars))
    vars_to_states = []
    for i, st in enumerate(states_to_vars):
        states_to_vars[st] = S.gen(i)
        vars_to_states.append(st)
    def R_wt(in_top, in_bot, out_top, out_bot):
        return states_to_vars.get((in_top,in_bot,out_top,out_bot), 0)

    states = list(cartesian_product([list(range(4)), list(range(4)),
                                     list(range(4))]))

    L1 = matrix(base, [[L_wt(s[0], s[2], t[0], t[2], zi) if s[1] == t[1]
                        else 0 for t in states] for s in states])
    L2 = matrix(base, [[L_wt(s[1], s[2], t[1], t[2], zj) if s[0] == t[0]
                        else 0 for t in states] for s in states])
    R = matrix(S, [[R_wt(s[0], s[1], t[1], t[0]) if s[2] == t[2] else 0
                    for t in states] for s in states])
    RLL = R*(L1*L2) - (L2*L1)*R

    M = matrix(base, [[RLL[i,j].monomial_coefficient(g) for g in S.gens()]
                      for i in range(len(states))
                      for j in range(len(states))])
    ker = [b for b in M.right_kernel().basis() if b[0] == 1]
    assert len(ker) == 1, len(ker)  # Safety check
    ret = {}
    for i, val in enumerate(ker[0]):
        if val != 0:
            ret[vars_to_states[i]] = SR(val).factor()
    return ret
\end{lstlisting}
Running this function results in
\begin{lstlisting}
sage: compute_R_matrix()
{(0, 0, 0, 0): 1,
 (0, 1, 0, 1): (beta*zj + 1)/(beta*zi + 1),
 (0, 2, 0, 2): (beta*zj + 1)/(beta*zi + 1),
 (0, 3, 0, 3): (beta*zj + 1)/(beta*zi + 1),
 (1, 0, 0, 1): -beta*(zi - zj)/(beta*zi + 1),
 (1, 0, 1, 0): 1,
 (1, 1, 1, 1): 1,
 (1, 2, 1, 2): (beta*zj + 1)/(beta*zi + 1),
 (1, 3, 1, 3): (beta*zj + 1)/(beta*zi + 1),
 (2, 0, 0, 2): -beta*(zi - zj)/(beta*zi + 1),
 (2, 0, 2, 0): 1,
 (2, 1, 1, 2): -beta*(zi - zj)/(beta*zi + 1),
 (2, 1, 2, 1): 1,
 (2, 2, 2, 2): 1,
 (2, 3, 2, 3): (beta*zj + 1)/(beta*zi + 1),
 (3, 0, 0, 3): -beta*(zi - zj)/(beta*zi + 1),
 (3, 0, 3, 0): 1,
 (3, 1, 1, 3): -beta*(zi - zj)/(beta*zi + 1),
 (3, 1, 3, 1): 1,
 (3, 2, 2, 3): -beta*(zi - zj)/(beta*zi + 1),
 (3, 2, 3, 2): 1,
 (3, 3, 3, 3): 1}
\end{lstlisting}

\bibliographystyle{alpha} 
\bibliography{models}

\newcommand{\etalchar}[1]{$^{#1}$}
\begin{thebibliography}{BBBG19b}

\bibitem[And11]{Anderson11}
Dave Anderson.
\newblock Introduction to equivariant cohomology in algebraic geometry.
\newblock Notes on lectures by W. Fulton at IMPAGNA summer school, 2010,
  \arxiv{1112.1421}, 2011.

\bibitem[And19]{Anderson19}
Dave Anderson.
\newblock K-theoretic chern class formulas for vexillary degeneracy loci.
\newblock {\em Adv. Math.}, 350:440--485, 2019.

\bibitem[BB93]{BB93}
Nantel Bergeron and Sara Billey.
\newblock R{C}-graphs and {S}chubert polynomials.
\newblock {\em Experiment. Math.}, 2(4):257--269, 1993.

\bibitem[BBB19]{BBB}
Ben Brubaker, Valentin Buciumas, and Daniel Bump.
\newblock A {Y}ang-{B}axter equation for metaplectic ice.
\newblock {\em Commun. Number Theory Phys.}, 13(1):101--148, 2019.

\bibitem[BBBG19a]{BBBG19II}
Ben Brubaker, Valentin Buciumas, Daniel Bump, and Nathan Gray.
\newblock Duality for metaplectic ice.
\newblock {\em Commun. Number Theory Phys.}, 13(1):101--148, 2019.

\bibitem[BBBG19b]{BBBGIwahori}
Ben Brubaker, Valentin Buciumas, Daniel Bump, and Henrik P.~A. Gustafsson.
\newblock Colored vertex models and {I}wahori {W}hittaker functions.
\newblock Preprint, \arxiv{1906.04140}, 2019.

\bibitem[BBBG19c]{BBBG19}
Ben Brubaker, Valentin Buciumas, Daniel Bump, and Henrik P.~A. Gustafsson.
\newblock Coloured five-vertex models and {D}emazure atoms.
\newblock Preprint, \arxiv{1902.01795}, 2019.

\bibitem[BBC{\etalchar{+}}12]{BBCFG12}
Ben Brubaker, Daniel Bump, Gautam Chinta, Solomon Friedberg, and Paul~E.
  Gunnells.
\newblock Metaplectic ice.
\newblock In {\em Multiple {D}irichlet series, {L}-functions and automorphic
  forms}, volume 300 of {\em Progr. Math.}, pages 65--92.
  Birkh\"auser/Springer, New York, 2012.

\bibitem[BBCG12]{BBCG12}
Ben Brubaker, Daniel Bump, Gautam Chinta, and Paul~E. Gunnells.
\newblock Crystals of type {B} and metaplectic whittaker functions.
\newblock In {\em Multiple {D}irichlet series, {L}-functions and automorphic
  forms}, volume 300 of {\em Progr. Math.}, pages 93--118.
  Birkh\"auser/Springer, New York, 2012.

\bibitem[BBF11]{BBF11}
Ben Brubaker, Daniel Bump, and Solomon Friedberg.
\newblock Schur polynomials and the {Y}ang-{B}axter equation.
\newblock {\em Comm. Math. Phys.}, 308(2):281--301, 2011.

\bibitem[BFH{\etalchar{+}}20]{frozenpipes}
Ben Brubaker, Claire Frechette, Andrew Hardt, Emily Tibor, and Katherine Weber.
\newblock Frozen pipes: {L}attice models for {G}rothendieck polynomials.
\newblock Preprint, \arxiv{2007.04310}, 2020.

\bibitem[Bor17]{Borodin17}
Alexei Borodin.
\newblock On a family of symmetric rational functions.
\newblock {\em Adv. Math.}, 306:973--1018, 2017.

\bibitem[Bri04]{Brion04}
Michel Brion.
\newblock Lectures on the geometry of flag varieties.
\newblock In {\em Topics in Cohomological Studies of Algebraic Varieties},
  Trends in Mathematics, pages 33--85. Birkh\"auser Boston, Boston, MA, 2004.

\bibitem[BSW20]{BSW20}
Valentin Buciumas, Travis Scrimshaw, and Katherine Weber.
\newblock Colored five-vertex models and {L}ascoux polynomials and atoms.
\newblock {\em J. Lond. Math. Soc.}, 2020.
\newblock To appear, \doi{10.1112/jlms.12347}.

\bibitem[Buc02]{Buch02}
Anders~Skovsted Buch.
\newblock A {L}ittlewood-{R}ichardson rule for the {$K$}-theory of
  {G}rassmannians.
\newblock {\em Acta Math.}, 189(1):37--78, 2002.

\bibitem[BW19]{BorodinWheelernsMac}
Alexei Borodin and Michael Wheeler.
\newblock Nonsymmetric {M}acdonald polynomials via integrable vertex models.
\newblock Preprint, \arxiv{1904.06804}, 2019.

\bibitem[CP94]{CPbook}
Vyjayanthi Chari and Andrew Pressley.
\newblock {\em A guide to quantum groups}.
\newblock Cambridge University Press, Cambridge, 1994.

\bibitem[CP16]{CP16}
Ivan Corwin and Leonid Petrov.
\newblock Stochastic higher spin vertex models on the line.
\newblock {\em Comm. Math. Phys.}, 343(2):651--700, 2016.

\bibitem[EKLP92]{EKLP92}
Noam Elkies, Greg Kuperberg, Michael Larsen, and James Propp.
\newblock Alternating-sign matrices and domino tilings. {II}.
\newblock {\em J. Algebraic Combin.}, 1(3):219--234, 1992.

\bibitem[FK94]{FK94}
Sergey Fomin and Anatol~N. Kirillov.
\newblock Grothendieck polynomials and the {Y}ang-{B}axter equation.
\newblock In {\em Formal power series and algebraic combinatorics/{S}\'eries
  formelles et combinatoire alg\'ebrique}, pages 183--189. DIMACS, Piscataway,
  NJ, 1994.

\bibitem[FK96]{FK96}
Sergey Fomin and Anatol~N. Kirillov.
\newblock The {Y}ang-{B}axter equation, symmetric functions, and {S}chubert
  polynomials.
\newblock In {\em Proceedings of the 5th {C}onference on {F}ormal {P}ower
  {S}eries and {A}lgebraic {C}ombinatorics ({F}lorence, 1993)}, volume 153,
  pages 123--143, 1996.

\bibitem[Ful92]{Fulton92}
William Fulton.
\newblock Flags, {S}chubert polynomials, degeneracy loci, and determinantal
  formulas.
\newblock {\em Duke Math. J.}, 65(3):381--420, 1992.

\bibitem[GK15]{GK15}
William Graham and Victor Kreiman.
\newblock Excited {Y}oung diagrams, equivariant {$K$}-theory, and {S}chubert
  varieties.
\newblock {\em Trans. Amer. Math. Soc.}, 367(9):6597--6645, 2015.

\bibitem[GK17]{GK17}
Vassily Gorbounov and Christian Korff.
\newblock Quantum integrability and generalised quantum {S}chubert calculus.
\newblock {\em Adv. Math.}, 313:282--356, 2017.

\bibitem[Gra17]{Gray17}
Nathan Gray.
\newblock Metaplectic ice for {C}artan type {C}.
\newblock Preprint, \arxiv{1709.04971}, 2017.

\bibitem[GV85]{GV85}
Ira Gessel and G\'erard Viennot.
\newblock Binomial determinants, paths, and hook length formulae.
\newblock {\em Adv. in Math.}, 58(3):300--321, 1985.

\bibitem[HK05]{HK05}
Angie Hamel and Ronald King.
\newblock U-turn alternating sign matrices, symplectic shifted tableaux and
  their weighted enumeration.
\newblock {\em J. Algebraic Combin.}, 21(4):395--421, 2005.

\bibitem[HKZJ19]{HKZJ18}
Iva Halacheva, Allen Knutson, and Paul Zinn-Justin.
\newblock Restricting {S}chubert classes to symplectic {G}rassmannians using
  self-dual puzzles.
\newblock {\em S\'{e}m. Lothar. Combin.}, 82B:Art. 83, 12 pp., 2019.

\bibitem[HM18]{HM18}
Thomas Hudson and Tomoo Matsumura.
\newblock Vexillary degeneracy loci classes in {$K$}-theory and algebraic
  cobordism.
\newblock {\em European J. Combin.}, 70:190--201, 2018.

\bibitem[HPW20]{HPW20}
Zachary Hamaker, Oliver Pechenik, and Anna Weigandt.
\newblock Gr{\"o}bner geometry of {S}chubert polynomials through ice.
\newblock Preprint, \arxiv{2003.13719}, 2020.

\bibitem[Hud14]{Hudson14}
Thomas Hudson.
\newblock A {T}hom-{P}orteous formula for connective {$K$}-theory using
  algebraic cobordism.
\newblock {\em J. K-Theory}, 14(2):343--369, 2014.

\bibitem[IN09]{IN09}
Takeshi Ikeda and Hiroshi Naruse.
\newblock Excited {Y}oung diagrams and equivariant {S}chubert calculus.
\newblock {\em Trans. Amer. Math. Soc.}, 361(10):5193--5221, 2009.

\bibitem[IN13]{IN13}
Takeshi Ikeda and Hiroshi Naruse.
\newblock {$K$}-theoretic analogues of factorial {S}chur {$P$}- and
  {$Q$}-functions.
\newblock {\em Adv. Math.}, 243:22--66, 2013.

\bibitem[IS14]{IS14}
Takeshi Ikeda and Tatsushi Shimazaki.
\newblock A proof of {$K$}-theoretic {L}ittlewood-{R}ichardson rules by
  {B}ender-{K}nuth-type involutions.
\newblock {\em Math. Res. Lett.}, 21(2):333--339, 2014.

\bibitem[Iva12]{Ivanov12}
Dmitriy Ivanov.
\newblock Symplectic ice.
\newblock In {\em Multiple {D}irichlet series, {L}-functions and automorphic
  forms}, volume 300 of {\em Progr. Math.}, pages 205--222.
  Birkh\"{a}user/Springer, New York, 2012.

\bibitem[KL04]{KL04}
Victor Kreiman and V.~Lakshmibai.
\newblock Multiplicities of singular points in {S}chubert varieties of
  {G}rassmannians.
\newblock In {\em Algebra, arithmetic and geometry with applications ({W}est
  {L}afayette, {IN}, 2000)}, pages 553--563. Springer, Berlin, 2004.

\bibitem[KM04]{KM04}
Allen Knutson and Ezra Miller.
\newblock Subword complexes in {C}oxeter groups.
\newblock {\em Adv. Math.}, 184(1):161--176, 2004.

\bibitem[KM05]{KM05}
Allen Knutson and Ezra Miller.
\newblock Gr\"{o}bner geometry of {S}chubert polynomials.
\newblock {\em Ann. of Math. (2)}, 161(3):1245--1318, 2005.

\bibitem[KMO15]{KMO15}
Atsuo Kuniba, Shouya Maruyama, and Masato Okado.
\newblock Multispecies {TASEP} and combinatorial {$R$}.
\newblock {\em J. Phys. A}, 48(34):34FT02, 19, 2015.

\bibitem[KMO16a]{KMO16}
Atsuo Kuniba, Shouya Maruyama, and Masato Okado.
\newblock Inhomogeneous generalization of a multispecies totally asymmetric
  zero range process.
\newblock {\em J. Stat. Phys.}, 164(4):952--968, 2016.

\bibitem[KMO16b]{KMO16II}
Atsuo Kuniba, Shouya Maruyama, and Masato Okado.
\newblock Multispecies {TASEP} and the tetrahedron equation.
\newblock {\em J. Phys. A}, 49(11):114001, 22, 2016.

\bibitem[KMY09]{KMY09}
Allen Knutson, Ezra Miller, and Alexander Yong.
\newblock Gr\"obner geometry of vertex decompositions and of flagged tableaux.
\newblock {\em J. Reine Angew. Math.}, 630:1--31, 2009.

\bibitem[Kra01]{Krattenthaler01}
C.~Krattenthaler.
\newblock On multiplicities of points on {S}chubert varieties in
  {G}rassmannians.
\newblock {\em S\'{e}m. Lothar. Combin.}, 45:Art. B45c, 11, 2000/01.

\bibitem[Kra05]{Krattenthaler05}
C.~Krattenthaler.
\newblock On multiplicities of points on {S}chubert varieties in {G}ra\ss
  mannians. {II}.
\newblock {\em J. Algebraic Combin.}, 22(3):273--288, 2005.

\bibitem[Kre05]{Kreiman05}
Victor Kreiman.
\newblock Schubert classes in the equivariant {$K$}-theory and equivariant
  cohomology of the {G}rassmannian.
\newblock Preprint, \arxiv{math/0512204}, 2005.

\bibitem[Kup96]{Kuperberg96}
Greg Kuperberg.
\newblock Another proof of the alternating-sign matrix conjecture.
\newblock {\em Internat. Math. Res. Notices}, (3):139--150, 1996.

\bibitem[KZJ17]{KZJ17}
Allen Knutson and Paul Zinn-Justin.
\newblock Schubert puzzles and integrability {I}: invariant trilinear forms.
\newblock Preprint, \arxiv{1706.10019}, 2017.

\bibitem[Las02]{Lascoux02}
Alain Lascoux.
\newblock Chern and {Y}ang through ice.
\newblock Preprint, 2002.

\bibitem[Lin73]{Lindstrom73}
Bernt Lindstr\"om.
\newblock On the vector representations of induced matroids.
\newblock {\em Bull. London Math. Soc.}, 5:85--90, 1973.

\bibitem[LLS20]{LLS18}
Thomas Lam, Seung~Jin Lee, and Mark Shimozono.
\newblock Back stable {S}chubert calculus.
\newblock {\em Compos. Math.}, 2020.
\newblock To appear, \arxiv{1806.11233}.

\bibitem[LRS06]{LRS06}
V.~Lakshmibai, K.~N. Raghavan, and P.~Sankaran.
\newblock Equivariant {G}iambelli and determinantal restriction formulas for
  the {G}rassmannian.
\newblock {\em Pure Appl. Math. Q.}, 2(3, Special Issue: In honor of Robert D.
  MacPherson. Part 1):699--717, 2006.

\bibitem[LS82]{LS82}
Alain Lascoux and Marcel-Paul Sch{\"u}tzenberger.
\newblock Structure de {H}opf de l'anneau de cohomologie et de l'anneau de
  {G}rothendieck d'une vari\'et\'e de drapeaux.
\newblock {\em C. R. Acad. Sci. Paris S\'er. I Math.}, 295(11):629--633, 1982.

\bibitem[Mat17]{Matsumura17}
Tomoo Matsumura.
\newblock An algebraic proof of determinant formulas of {G}rothendieck
  polynomials.
\newblock {\em Proc. Japan Acad. Ser. A Math. Sci.}, 93(8):82--85, 2017.

\bibitem[McN06]{McNamara06}
Peter~J. McNamara.
\newblock Factorial {G}rothendieck polynomials.
\newblock {\em Electron. J. Combin.}, 13(1):Research Paper 71, 40, 2006.

\bibitem[Mon16]{Monical16}
Cara Monical.
\newblock Set-valued skyline fillings.
\newblock Preprint, \arxiv{1611.08777}, 2016.

\bibitem[Mot20]{Motegi20}
Kohei Motegi.
\newblock Integrability approach to
  {F}eh\'{e}r-{N}\'{e}methi-{R}im\'{a}nyi-{G}uo-{S}un type identities for
  factorial {G}rothendieck polynomials.
\newblock {\em Nuclear Phys. B}, 954:114998, 2020.

\bibitem[MPS20]{MPS18}
Cara Monical, Oliver Pechenik, and Travis Scrimshaw.
\newblock Crystal structures for symmetric {G}rothendieck polynomials.
\newblock {\em Transform. Groups}, 2020.
\newblock To appear, \doi{10.1007/s00031-020-09623-y}.

\bibitem[MS13]{MS13}
Kohei Motegi and Kazumitsu Sakai.
\newblock Vertex models, {TASEP} and {G}rothendieck polynomials.
\newblock {\em J. Phys. A}, 46(35):355201, 26, 2013.

\bibitem[MS14]{MS14}
Kohei Motegi and Kazumitsu Sakai.
\newblock {$K$}-theoretic boson-fermion correspondence and melting crystals.
\newblock {\em J. Phys. A}, 47(44):445202, 2014.

\bibitem[MS20]{MS19}
Tomoo Matsumura and Shogo Sugimoto.
\newblock Factorial flagged {G}rothendieck poylnomials.
\newblock In {\em Proceedings of the conference ``{A}n international festival
  in {S}chubert {C}alculus'' held in {G}uangzhou in {N}ovember 2017}, Springer
  Proceedings in Mathematics {\&} Statistics. Springer, New York, 2020.
\newblock To appear, \arxiv{1903.02169}.

\bibitem[Sag01]{Sagan01}
Bruce~E. Sagan.
\newblock {\em The symmetric group}, volume 203 of {\em Graduate Texts in
  Mathematics}.
\newblock Springer-Verlag, New York, second edition, 2001.
\newblock Representations, combinatorial algorithms, and symmetric functions.

\bibitem[Sag20]{sage}
The Sage Developers.
\newblock {\em {S}age {M}athematics {S}oftware ({V}ersion 9.1)}, 2020.
\newblock \url{http://www.sagemath.org}.

\bibitem[Wac85]{Wachs85}
Michelle~L. Wachs.
\newblock Flagged {S}chur functions, {S}chubert polynomials, and symmetrizing
  operators.
\newblock {\em J. Combin. Theory Ser. A}, 40(2):276--289, 1985.

\bibitem[Wei20]{Weigandt20}
Anna Weigandt.
\newblock Bumpless pipe dreams and alternating sign matrices.
\newblock Preprint, \arxiv{2003.07342}, 2020.

\bibitem[WZJ19]{WZJ19}
Michael Wheeler and Paul Zinn-Justin.
\newblock Littlewood-{R}ichardson coefficients for {G}rothendieck polynomials
  from integrability.
\newblock {\em J. Reine Angew. Math.}, 757:159--195, 2019.

\end{thebibliography}

\end{document}